\documentclass[reqno]{amsart}
\usepackage{hyperref}
\usepackage{graphicx}
\usepackage[latin1]{inputenc}
\usepackage[english,activeacute]{babel}
\usepackage{esvect}
\numberwithin{equation}{section}
\newtheorem{theorem}{Theorem}[section]
\newtheorem{lemma}[theorem]{Lemma}
\newtheorem{remark}[theorem]{Remark}

\newtheorem{proposition}[theorem]{Proposition}
\newtheorem{corollary}[theorem]{Corollary}

\title[Orbital Stability of standing waves]
{Orbital stability of standing waves for a system of nonlinear
 Schr\"{o}dinger equations with three wave interaction}

\author[Alex H. Ardila]{}
\email{ardila@impa.br}


\begin{document}
\maketitle

\centerline{\scshape Alex H. Ardila}
\medskip
{\footnotesize
 \centerline{Instituto Nacional de Matem\'atica Pura e Aplicada - IMPA,}
\centerline{Estrada Dona Castorina 110, CEP 22460-320, Rio de Janeiro, RJ, Brazil.}
} 

\begin{abstract}
We study the existence and stability of standing waves solutions of a three-coupled nonlinear Schr\"{o}dinger system related to the Raman amplification in a plasma. By means of the concentration-compacteness method, we provide a  characterization of the standing waves solutions as minimizers of an energy functional subject to three independent $L^{2}$ mass constraints. As a consequence, we establish existence and  orbital stability of solitary waves. 
\end{abstract}

\section{Introduction}
\label{S:0}
In this paper, our purpose is to investigate the orbital stability of standing waves solutions for the following system of nonlinear Schr\"{o}dinger equations with three-wave interaction
\begin{equation}\label{NLS}
\begin{cases} 
 i\partial_{t}u_{1}+\partial^{2}_{x}u_{1}+\beta|u_{1}|^{p-1}u_{1}=-\alpha u_{3}\overline{u_{2}},  \\
 i\partial_{t}u_{2}+\partial^{2}_{x}u_{2}+\beta|u_{2}|^{p-1}u_{2}=-\alpha u_{3}\overline{u_{1}},\\   
 i\partial_{t}u_{3}+\partial^{2}_{x}u_{3}+\beta|u_{3}|^{p-1}u_{3}=-\alpha u_{1}{u_{2}}, 
\end{cases} 
\end{equation} 
for $(x,t)\in \mathbb{R}^{2}$, where $p>1$, $\beta>0$, $\alpha>0$, and $u_{1}$, $u_{2}$ and $u_{3}$  are complex valued functions of 
$(x,t)\in \mathbb{R}^{2}$. The system \eqref{NLS} was introduced by M. Colin, T. Colin and M. Ohta in \cite{CCaH} as a simplified model of a quasilinear Zakharov system related to the Raman amplification in a plasma analyzed in \cite{RP1, RP2}.

In what follows, we use the vectorial notation $\vv{u}=(u_{1}, u_{2}, u_{3})$. Formally, the system \eqref{NLS} has  the following three conserved quantities. The first conserved quantity  is the energy $E$ defined by
\begin{equation}\label{En}
E(\vv{u})=\sum^{3}_{i=1}\left\{\frac{1}{2}\int_{\mathbb{R}}|\partial_{x}u_{i}|^{2}\,dx-\frac{\beta}{p+1}\int_{\mathbb{R}}|u_{i}|^{p+1}\,dx\right\}-\alpha\,\Re\int_{\mathbb{R}}u_{1}u_{2}\overline{u_{3}}\, dx,
\end{equation}
where $\Re(z)$ denotes the real part of a complex number $z$. Other two conserved quantities  are
\begin{equation}\label{Mas}
Q_{1}(\vv{u})=\|u_{1}\|^{2}_{L^{2}}+\|u_{3}\|^{2}_{L^{2}}\quad \text{and}\quad Q_{2}(\vv{u})=\|u_{2}\|^{2}_{L^{2}}+\|u_{3}\|^{2}_{L^{2}}.
\end{equation}

We see that the well-posedness of the Cauchy Problem for \eqref{NLS} in $H^{1}(\mathbb{R}; \mathbb{C}^{3})$ and the conservation laws follow from the standard regularizing argument; see Chapter 4 in \cite{CB} for more details. 

\begin{proposition} \label{WP}
Let $1<p< +\infty$ and $\alpha>0$. For every  $\vv{u}_{0}\in H^{1}(\mathbb{R}; \mathbb{C}^{3})$, there is $T_{{\rm max}}=T_{{\rm max}}(\vv{u}_{0})>0$ and a unique solution $\vv{u}\in C([0,T_{{\rm max}}),H^{1}(\mathbb{R}; \mathbb{C}^{3}))$  of \eqref{NLS} with $\vv{u}(0)=\vv{u}_{0}$ such that either $T_{{\rm max}}=+\infty$
or $T_{{\rm max}}<+\infty$ and $\lim_{t\rightarrow T_{{\rm max}}}\|\partial_{x}\vv{u}(t)\|^{2}_{L^{2}}=+\infty$. 
Furthermore, the solution $\vv{u}$ satisfies the conservation laws: for all $t\in [0, T^{\rm max})$,
\begin{equation*}
E(\vv{u}(t))=E(\vv{u}_{0}), \,\,\,\,   Q_{1}(\vv{u}(t))=Q_{1}(\vv{u}_{0})\,\,\,\,   \text{and}\,\,\,\, Q_{2}(\vv{u}(t))=Q_{2}(\vv{u}_{0}).
\end{equation*}
\end{proposition}

Notice that if $1<p<5$, then the Cauchy problem of \eqref{NLS} is globally well posed in $H^{1}(\mathbb{R}; \mathbb{C}^{3})$. Indeed, assume that $T_{{\rm max}}<+\infty$ and therefore $\lim_{t\rightarrow T_{{\rm max}}}\|\partial_{x}\vv{u}(t)\|^{2}_{L^{2}}=+\infty$.
From the Gagliardo-Nirenberg inequality, H\"older inequality and the conservation laws, there exists a constant $C>0$ such that for all $\vv{u}_{}\in H^{1}(\mathbb{R}; \mathbb{C}^{3})$ (see \eqref{E1} below for more detail)
\begin{align}\label{GC1} 
\|u_{i}\|^{p+1}_{L^{p+1}}&\leq C\|\partial_{x}u_{i}\|^{\frac{(p-1)}{2}}_{L^{2}}\|u_{i}\|^{p+1-\frac{(p-1)}{2}}_{L^{2}}  \quad \text{for $i=1$, $2$, $3$,}\\   \label{GC2}
\Re\int_{\mathbb{R}}u_{1}(t)u_{2}(t)\overline{u_{3}(t)}\,& dx\leq C\sum^{3}_{i=1} \|\partial_{x}u_{i}(t)\|^{\frac{1}{2}}_{L^{2}}\quad \text{for all $t\in [0, T_{{\rm max}})$}.
\end{align}
Combining \eqref{GC1} and \eqref{GC2} leads to
\begin{equation}\label{CBC}
E(\vv{u}_{0})=E(\vv{u}(t))\geq \sum^{3}_{i=1} \|\partial_{x}u_{i}(t)\|^{2}_{L^{2}}\left(1-C\left(\|\partial_{x}u_{i}(t)\|^{\frac{(p-1)}{2}-2}_{L^{2}}+\|\partial_{x}u_{i}(t)\|^{-\frac{3}{2}}_{L^{2}}\right)\right),
\end{equation}
since $p<5$, we see that $\frac{(p-1)}{2}-2<0$, and thus letting $\|\partial_{x}\vv{u}(t)\|^{2}_{L^{2}}\rightarrow +\infty$
when $t\rightarrow T_{{\rm max}}$ leads to a contradiction in \eqref{CBC}.

For every $\theta_{1}$, $\theta_{2}\in \mathbb{R}$ and $y\in \mathbb{R}^{N}$ we define $R(\theta_{1}, \theta_{2})$ and $\tau_{y}$ by 
\begin{equation*}
R(\theta_{1}, \theta_{2})\vv{u}=(e^{i\theta_{1}}u_{1}, e^{i\theta_{2}}u_{2}, e^{i(\theta_{1}+\theta_{2})}u_{3}), \quad \tau_{y}\vv{u}(x)=\vv{u}(x-y),
\end{equation*}
for all $\vv{u}\in H^{1}(\mathbb{R}; \mathbb{C}^{3})$. Note that \eqref{NLS} can be written as
\begin{equation*}
\partial_{t}\vv{u}(t)=-iE^{\prime}(\vv{u}(t))
\end{equation*}
and that $E(R(\theta_{1}, \theta_{2})\tau_{y}\vv{u})=E(\vv{u})$ for all $\theta_{1}$, $\theta_{2}\in \mathbb{R}$, $y\in \mathbb{R}^{N}$ and 
$\vv{u}\in H^{1}(\mathbb{R}; \mathbb{C}^{3})$. For $\omega_{1}$, $\omega_{2}\in \mathbb{R}$ we define the action $K_{\omega_{1}, \omega_{1}}:  H^{1}(\mathbb{R}; \mathbb{C}^{3})\rightarrow \mathbb{R}$ by 
\begin{equation*}
K_{\omega_{1}, \omega_{1}}(\vv\phi)=E(\vv\phi)+\omega_{1}Q_{1}(\vv\phi)+\omega_{2}Q_{2}(\vv\phi).
\end{equation*}
We remark that the Euler-Lagrange equation  $K^{\prime}_{\omega_{1}, \omega_{1}}(\vv\phi)=0$ is written as
\begin{equation}\label{SP}
\begin{cases} 
-\partial^{2}_{x}\phi_{1}+2\omega_{1}\phi_{1}-\beta|\phi_{1}|^{p-1}\phi_{1}=\alpha \phi_{3}\overline{\phi_{2}},  \\
-\partial^{2}_{x}\phi_{2}+2\omega_{2}\phi_{2}-\beta|\phi_{2}|^{p-1}\phi_{2}=\alpha\phi_{3}\overline{\phi_{1}},\\   
-\partial^{2}_{x}\phi_{3}+2\omega_{3}\phi_{3}-\beta|\phi_{3}|^{p-1}\phi_{3}=\alpha \phi_{1}\phi_{2}, 
\end{cases} 
\end{equation}
where $\omega_{3}=\omega_{1}+\omega_{2}$, and that if $K^{\prime}_{\omega_{1}, \omega_{1}}(\vv\phi)=0$, then $\vv{u}(x,t)=R(2\omega_{1}t, 2\omega_{2}t)\vv{\phi}(x)$ is a solution of \eqref{NLS}. From the physical as well as the mathematical point of view, a central issue is to study the existence and stability of standing  waves  of system  \eqref{NLS}. A standing  waves  solution of \eqref{NLS} is a solution of the form 
$\vv{u}(x,t)=R(2\omega_{1}t, 2\omega_{2}t)\vv{\phi}(x)=(e^{2i\omega_{1}t}\phi_{1}(x), e^{2i\omega_{2}t}\phi_{2}(x), e^{2i(\omega_{1}+\omega_{2})t}\phi_{3}(x))$, where $(\omega_{1}, \omega_{2})\in \mathbb{R}^{2}$ and  $(\phi_{1}, \phi_{2},\phi_{3})$ are complex valued functions which have to satisfy the  system of ordinary differential equations \eqref{SP}. 

Previously, Pomponio \cite{Apo} had proved the existence of vector solutions of \eqref{SP} as minimizers of the action $K_{\omega_{1}, \omega_{1}}$ on the Nehari manifold. More specifically, it was shown in \cite{Apo} that vector solutions of \eqref{SP} exist whenever value of the coupling parameter $\alpha$ is sufficiently large. In this paper our approach is different and is based on  the concentration compactness method of P.L. Lions \cite{PLL1}. Given any $\gamma>0$, $\mu>0$ and $s>0$ we look for solutions $(\phi, \varphi, \psi)\in H^{1}(\mathbb{R}; \mathbb{C}^{3})$ of \eqref{SP} satisfying the  condition $\|\phi\|^{2}_{L^{2}}=\gamma$, $\|\varphi\|^{2}_{L^{2}}=\mu$ and $\|\psi\|^{2}_{L^{2}}=s$.  Such solutions are of interest in physics and sometimes referred to as normalized solutions. With this in mind, we consider the following variational problem
\begin{equation}\label{MP}
I(\gamma, \mu,s):=\inf\left\{E(\vv{u}): \vv{u}\in H^{1}(\mathbb{R}; \mathbb{C}^{3}),\,\,
 \|u_{1}\|^{2}_{L^{2}}=\gamma,\,\,\, \|u_{2}\|^{2}_{L^{2}}=\mu,\,\,\, \|u_{3}\|^{2}_{L^{2}}=s\right\}.
\end{equation}
We denote the set of nontrivial minimizers for $I(\gamma, \mu,s)$ by 
\begin{align*}
\mathcal{G}_{\gamma, \mu,s}&=\left\{\vv{u}\in H^{1}(\mathbb{R}; \mathbb{C}^{3}):E(\vv{u})=I(\gamma, \mu,s) \,\, \text{such that}\right.\\
&\left.\|u_{1}\|^{2}_{L^{2}}=\gamma,\,\,\,\, \|u_{2}\|^{2}_{L^{2}}=\mu\,\, \text{and}\,\,\|u_{3}\|^{2}_{L^{2}}=s\right\}.
\end{align*}

Before establishing our first result of existence, we define a minimizing sequence for $I(\gamma, \mu, s)$ to be a sequence $\left\{\vv{u}_{n}\right\}$ in $H^{1}(\mathbb{R}; \mathbb{C}^{3})$ such that  $\|u_{1, n}\|^{2}_{L^{2}}\rightarrow \gamma$,  $\|u_{2, n}\|^{2}_{L^{2}}\rightarrow \mu$, $\|u_{3, n}\|^{2}_{L^{2}}\rightarrow s$ and $E(\vv{u}_{n})\rightarrow I(\gamma, \mu, s)$ as $n$ goes to $+\infty$; this convention will be useful later, in the proof of the Theorem \ref{NTE} below.

\begin{theorem} \label{SPI}
Suppose $\gamma>0$, $\mu>0$, $s>0$ and $1<p<5$. Then the following properties hold.\\
{\rm (i)} The set $\mathcal{G}_{\gamma, \mu,s}$ is not empty. Furthermore, any minimizing  sequence $\left\{\vv{u}_{n}\right\}$ of  $I(\gamma, \mu,s)$ is relatively compact in $H^{1}(\mathbb{R}; \mathbb{C}^{3})$ up to translations. That is, there exist $\left\{y_{n}\right\}\subset \mathbb{R}$ and an element $\vv{\varphi}$  such that  $\left\{\vv{u}_{n}(\cdot+y_{n})\right\}$ has a subsequence converging strongly in $H^{1}(\mathbb{R}; \mathbb{C}^{3})$ to $\vv{\varphi}$. \\
{\rm (ii)}
\begin{equation}\label{ABCD}
\lim_{n\rightarrow\infty}\inf_{\vv{g}\in\mathcal{G}_{\gamma, \mu,s}, y\in\mathbb{R}}\|\vv{u}_{n}(\cdot+y)-\vv{g}\|_{H^{1}(\mathbb{R}; \mathbb{C}^{3})}=0. 
\end{equation}
{\rm (iii)}
\begin{equation}\label{AB234}
\lim_{n\rightarrow\infty}\inf_{\vv{g}\in\mathcal{G}_{\gamma, \mu,s}}\|\vv{u}_{n}-\vv{g}\|_{H^{1}(\mathbb{R}; \mathbb{C}^{3})}=0.
\end{equation} 
{\rm (iv)} Each function $\vv{u}\in \mathcal{G}_{\gamma, \mu,s}$ is a classical solution of the system \eqref{SP} for some $\omega_{1}$, $\omega_{2}$, $\omega_{3}\in \mathbb{R}$. Furthermore, there exist numbers $\theta_{1}$, $\theta_{2}\in \mathbb{R}$ and non-negative functions $\rho_{j}$ such that $u_{1}(x)=e^{i\theta_{1}}\rho_{1}(x)$, $u_{1}(x)=e^{i\theta_{2}}\rho_{2}(x)$ and $u_{3}(x)=e^{i(\theta_{1}+\theta_{2})}\rho_{3}(x)$ for all $x\in \mathbb{R}$.
\end{theorem}
We remark that a classical solution of \eqref{SP} is a function $\vv{u}\in H^{1}(\mathbb{R}; \mathbb{C}^{3})$ with $u_{i}\in C^{2}(\mathbb{R})$, which solves \eqref{SP} in the classical sense (that is, using the classical notion of derivative).

Theorem \ref{SPI} is obtained via variational approach and using  the concentration compactness method of P.L. Lions \cite{PLL1}.  Similar techniques have been used previously in \cite{Albert} (see also \cite{BBOtra, BB2016, NZW}) to prove the existence and orbital stability of standing wave solutions to NLS-KdV systems. Notice that if $(\gamma_{1},\mu_{1}, s_{1})\neq (\gamma_{2},\mu_{2}, s_{2})$, then the sets $\mathcal{G}_{\gamma_{1}, \mu_{1},s_{1}}$ and $\mathcal{G}_{\gamma_{2}, \mu_{2},s_{2}}$ are disjoint; that is, the set of minimizers  $\mathcal{G}_{\gamma, \mu,s}$ forms a true three-parameter family.

For $p=2$ and $\gamma=\mu=s>0$, we have an explicit characterization of the set of minimizers $\mathcal{G}_{\gamma, \mu, s}$.

\begin{theorem} \label{CSW}
Let $p=2$. For any $\omega>0$ fixed,
\begin{equation*}
\mathcal{G}_{\gamma(\omega), \gamma(\omega), \gamma(\omega)}=\left\{\left(e^{i\theta_{1}}\psi_{\omega}(\cdot+y), e^{i\theta_{2}}\psi_{\omega}(\cdot+y), e^{i(\theta_{1}+\theta_{2})}\psi_{\omega}(\cdot+y)\right): \theta_{1}, \theta_{2}, y\in\mathbb{R}\right\},
\end{equation*}
where 
\begin{equation*}
\gamma(\omega)=\frac{12\sqrt{2}\,\omega^{3/2}}{(\alpha+\beta)^{2}}, \quad \psi_{\omega}(x)=\frac{3\omega}{(\alpha+\beta)}\mathrm{sech}^{2}\left(\frac{1}{2}\sqrt{2\omega}x\right).
\end{equation*}
\end{theorem}

The stability theory involves yet another variational formulation of standing waves solutions for \eqref{NLS}. For fixed $\gamma>0$ and $\mu>0$, let
\begin{equation}\label{Nvp}
J(\gamma, \mu)=\inf\left\{E(\vv{u}): \vv{u}\in H^{1}(\mathbb{R}; \mathbb{C}^{3}),\,\, Q_{1}(\vv{u})=\gamma \quad\text{and}\quad Q_{2}(\vv{u})=\mu \right\}.
\end{equation}
Following our convention, let us define a minimizing sequence for $J(\gamma, \mu)$ to be a sequence $\left\{\vv{u}_{n}\right\}$ in $H^{1}(\mathbb{R}; \mathbb{C}^{3})$ such that  $Q_{1}(\vv{u}_{n})\rightarrow \gamma$,  $Q_{2}(\vv{u}_{n})\rightarrow \mu$,  and $E(\vv{u}_{n})\rightarrow J(\gamma, \mu)$ as $n$ goes to $+\infty$.  The set of nontrivial minimizers for $J(\gamma, \mu)$ is
\begin{equation*}
 \mathcal{M}_{\gamma, \mu}=\inf\left\{ \vv{u}\in H^{1}(\mathbb{R}; \mathbb{C}^{3}) :E( \vv{u})=J(\gamma, \mu),
\,\,Q_{1}(\vv{u})=\gamma \quad\text{and}\quad Q_{2}(\vv{u})=\mu\right\}.
\end{equation*}

We remark that the family of variational problems associated to $J(\gamma, \mu)$ are suitable for studying the orbital stability  of standing waves solutions for \eqref{NLS} because both $E$ and the functionals $Q_{1}$ and $Q_{2}$  are invariant with regard to the flow generated by  \eqref{NLS}.

The following theorem gives the existence of a minimizer for $J(\gamma, \mu)$.

\begin{theorem} \label{NTE}
Suppose $1<p<5$. Then the following statements are true for all $\gamma>0$ and all $\mu>0$.\\
{\rm (i)}  Any minimizing sequence $\left\{\vv{u}_{n}\right\}$ of  $J(\gamma, \mu)$ is relatively compact in $H^{1}(\mathbb{R}; \mathbb{C}^{3})$ up to translations. In addition, the set $\mathcal{M}_{\gamma, \mu}$ is non-empty.\\
{\rm (ii)} The set of minimizers $\mathcal{M}_{\gamma, \mu}$ forms a true two-parameter family in the sense that
if $(\gamma_{1}, \mu_{1})\neq(\gamma_{2}, \mu_{2})$, then the sets $\mathcal{M}_{\gamma_{1}, \mu_{1}}$  and $\mathcal{M}_{\gamma_{2}, \mu_{2}}$ are disjoint.\\
{\rm (iii)} For every $\vv{u}\in \mathcal{M}_{\gamma, \mu}$, there exist numbers $\theta_{1}$, $\theta_{2}\in \mathbb{R}$ and functions $\rho_{1}(x)\geq 0$, $\rho_{2}(x)\geq 0$,  $\rho_{3}(x)\geq 0$, for all $x\in \mathbb{R}$,  such that $u_{1}(x)=e^{i\theta_{1}}\rho_{1}(x)$, $u_{2}(x)=e^{i\theta_{2}}\rho_{2}(x)$ and $u_{3}(x)=e^{i(\theta_{1}+\theta_{2})}\rho_{3}(x)$. 
\end{theorem}

It is clear that if $\vv{u}\in \mathcal{M}_{\gamma, \mu}$, then $(e^{2i\omega_{1}t}u_{1}(x), e^{2i\omega_{2}t}u_{2}(x), e^{2i(\omega_{1}+\omega_{2})t}u_{3}(x))$ is a standing wave solution of \eqref{NLS}, where the parameters $\omega_{1}$, $\omega_{2}\in \mathbb{R}$ appear as Lagrange multipliers. Furthermore, by part (iii) of the Theorem \ref{NTE} and from system \eqref{SP}, it is not hard to show (see Lemma \ref{LSwer} below) that if $\vv{u}\in \mathcal{M}_{\gamma, \mu}$ and $\gamma\neq\mu$, then $u_{i}\neq 0$, for $i=1$, $2$, $3$. In \cite{CCH2, CCaH}, the authors study the orbital stability/instability of standing waves with only one nonzero component. To be more specific, it was shown in \cite{CCH2, CCaH}  that $(e^{2i\omega t}\varphi, 0, 0)$ and $(0,  e^{2i\omega t}\varphi, 0)$ are stable, for every $\alpha>0$, while $(0, 0, e^{2i\omega t}\varphi)$ is  stable, if $0<\alpha<\alpha_{\ast}$, and it is  unstable,
if $\alpha>\alpha_{\ast}$, for a suitable constant $\alpha_{\ast}>0$. Here,  $\varphi\in H^{1}(\mathbb{R})$ is the unique
positive radial (least-energy) solution of 
\begin{equation*}
-\partial^{2}_{x}\varphi+2\omega\varphi-\beta|\varphi|^{p-1}\varphi=0\quad \text{in $\mathbb{R}$.}
\end{equation*}

As a corollary of the Theorem \ref{NTE}, we have that the set of minimizer $\mathcal{M}_{\gamma, \mu}$ is a stable set for the flow generated by system \eqref{NLS}.

\begin{corollary} \label{CLE}
Let $\gamma>0$ and  $\mu>0$. Then the set $\mathcal{M}_{\gamma, \mu}$ is $H^{1}(\mathbb{R}; \mathbb{C}^{3})$-stable for the flow generated by system \eqref{NLS} in the following sense. For arbitrary $\epsilon>0$, there exists $\delta>0$ such that, if $\vv{u}_{0}\in H^{1}(\mathbb{R}; \mathbb{C}^{3})$ verifies
\begin{equation*}
\inf_{\vv{\varphi}\in \mathcal{M}_{\gamma,\mu}} \|\vv{u}_{0}-\vv{\varphi}\|_{H^{1}(\mathbb{R}; \mathbb{C}^{3})}<\delta,
\end{equation*}
then the solution $\vv{u}(x,t)$ of the system \eqref{NLS} with the initial data $\vv{u}(x,0)=\vv{u}_{0}$ satisfies 
\begin{equation*}
\inf_{\vv{\varphi}\in \mathcal{M}_{\gamma, \mu}}\|\vv{u}(\cdot,t)-\vv{\varphi}\|_{H^{1}(\mathbb{R}; \mathbb{C}^{3})}<\epsilon, \quad \text{for all $t\geq 0$.}
\end{equation*}
\end{corollary}
The proof of Corollary \ref{CLE} makes only use of the conservation laws \eqref{En} and \eqref{Mas} and the compactness of any minimizing sequence for \eqref{Nvp}.

The rest of the paper is organized as follows. In Section \ref{S:1}, we prove Theorem \ref{SPI}. Theorem \ref{CSW} is proved in Section \ref{S:3/2}. Finally, Section \ref{S:2} contains the proof of Theorem \ref{NTE} and Corollary \ref{CLE}.

{\bf Notation.} We denote by $H_{\mathbb{C}}^{1}(\mathbb{R})$ the Sobolev space of all complex-valued functions $H^{1}(\mathbb{R},\mathbb{C})$,  its norm will be denoted by $\|\cdot\|_{H^{1}(\mathbb{R})}$.  The space of all real-valued  functions $f$ in $H_{\mathbb{C}}^{1}(\mathbb{R})$ will be denoted by $H^{1}(\mathbb{R})$. The space $L^{p}(\mathbb{R})$, denoted by $L^{p}$ for shorthand, is equipped with the norm $\|\cdot\|_{L^{p}}$. We also denote  by $H^{1}(\mathbb{R}; \mathbb{C}^{3})$ the product space $H^{1}(\mathbb{R}; \mathbb{C}^{3})=H^{1}_{\mathbb{C}}(\mathbb{R})\times H^{1}_{\mathbb{C}}(\mathbb{R})\times H^{1}_{\mathbb{C}}(\mathbb{R})$. Throughout this paper, the letter $C$ will frequently be used to denote various constants whose actual value is not important and which may vary from one line to the next.

\section{Existence of standing wave solutions}
\label{S:1}
In this section, we give the proof of Theorem \ref{SPI}. For simplicity, throughout this section we assume that $\beta=1$. All the results hold also for the case $\beta>0$, where minor straightforward modifications are required in the proofs. We have divided the proof of Theorem \ref{SPI} into a sequence of lemmas. 

\begin{lemma} \label{L1}
For all $\gamma$, $\mu$ and $s>0$, one has $-\infty< I(\gamma, \mu,s)<0$.
\end{lemma}
\begin{proof}
Let $ \vv{u}\in H^{1}(\mathbb{R}; \mathbb{C}^{3})$ such that $\|u_{1}\|^{2}_{L^{2}}=\gamma$, $\|u_{2}\|^{2}_{L^{2}}=\mu$ and $\|u_{3}\|^{2}_{L^{2}}=s$. To prove that $I(\gamma, \mu,s)>-\infty$ is suffices to bound $E(\vv{u})$ from below. Indeed, from the Gagliardo-Nirenberg inequality  \eqref{GC1}, Young's inequality  and H\"older's inequality we see that
\begin{align}\nonumber
\left|\Re\int_{\mathbb{R}}u_{1}u_{2}\overline{u_{3}}\, dx\right| &\leq \int_{\mathbb{R}}|u_{1}||u_{3}||{u_{3}}|\, dx\leq \|u_{1}\|_{L^{3}}\|u_{2}\|_{L^{3}}\|u_{3}\|_{L^{3}}\\
&\leq \frac{1}{3}\sum^{3}_{i=1} \|u_{i}\|^{3}_{L^{3}}\leq C\sum^{3}_{i=1} \|\partial_{x}u_{i}\|^{\frac{1}{2}}_{L^{2}}, \label{E1}
\end{align}
where $C$ is independent of $u_{1}$, $u_{2}$ and $u_{3}$. Since $p-1<4$, we have 
\begin{align*}
E(\vv{u})&=\sum^{3}_{i=1}\left\{\frac{1}{2}\int_{\mathbb{R}}|\partial_{x}u_{i}|^{2}\,dx-\frac{1}{p+1}\int_{\mathbb{R}}|u_{i}|^{p+1}\,dx\right\}-\alpha\,\Re\int_{\mathbb{R}}u_{1}u_{2}\overline{u_{3}}\, dx \\
&\geq \frac{1}{2}\sum^{3}_{i=1}\|\partial_{x}u_{i}\|^{2}_{L^{2}}-C\sum^{3}_{i=1}\|\partial_{x}u_{i}\|^{(p-1)/2}_{L^{2}}-C\sum^{3}_{i=1} \|\partial_{x}u_{i}\|^{\frac{1}{2}}_{L^{2}}>-\infty.
\end{align*}
On the other hand, to see that $I(\gamma, \mu,s)<0$, choose $\vv{u}\in H^{1}(\mathbb{R}; \mathbb{C}^{3})$ such that $\|u_{1}\|^{2}_{L^{2}}=\gamma$, $\|u_{2}\|^{2}_{L^{2}}=\mu$, $\|u_{3}\|^{2}_{L^{2}}=s$ and $u_{i}(x)>0$ for all $x\in \mathbb{R}$, $i=1$, $2$, $3$. For each $\theta>0$, set $\vv{u}_{\theta}=(u_{1,\theta}, u_{2,\theta}, u_{3,\theta})$, where $u_{i,\theta}(x)=\sqrt{\theta}u_{i}(\theta x)$. Then for all $\theta$, we see that $\|u_{i,\theta}\|^{2}_{L^{2}}=\|u_{i}\|^{2}_{L^{2}}$ and 
\begin{align*}
E(\vv{u}_{\theta})&=\sum^{3}_{i=1}\left\{\frac{1}{2}\int_{\mathbb{R}}|\partial_{x}u_{i,\theta}|^{2}\,dx-\frac{1}{p+1}\int_{\mathbb{R}}|u_{i,\theta}|^{p+1}\,dx\right\}-\alpha\,\int_{\mathbb{R}}u_{1,\theta}u_{2,\theta}{u_{3,\theta}}\, dx \\
&\leq  \frac{\theta^{2}}{2}\sum^{3}_{i=1}\|\partial_{x}u_{i}\|^{2}_{L^{2}}-\sqrt{\theta}\,\alpha\int_{\mathbb{R}}u_{1}u_{2}{u_{3}}\, dx,
\end{align*}
hence, by taking $\theta$ sufficiently small, we obtain $E(\vv{u}_{\theta})<0$.
\end{proof}

\begin{lemma} \label{DL1} 
Suppose  $\left\{\vv{u}_{n}\right\}=\left\{(u_{1,n},u_{2,n}, u_{3,n})\right\}$ is a minimizing sequence of problem $I(\gamma,\mu,s )$, then there exist constants $B>0$ and $\delta_{i}>0$ such that, for all sufficiently large $n$, \\
{\rm i)} $\sum^{3}_{i=1}\|\partial_{x}u_{i,n}\|^{2}_{L^{2}}\leq B$. \\
{\rm ii)} If $\gamma>0$, $\mu\geq 0$ and $s\geq 0$, then $\|\partial_{x}u_{1, n}\|^{}_{L^{2}}\geq \delta_{1}$ .\\
{\rm iii)} If $\gamma\geq0$, $\mu>0$ and  $s\geq 0$,  then $\|\partial_{x}u_{2, n}\|^{}_{L^{2}}\geq \delta_{2}$ .\\
{\rm iv)} If $\gamma\geq0$, $\mu\geq0$ and  $s> 0$,  then $\|\partial_{x}u_{3, n}\|^{}_{L^{2}}\geq \delta_{3}$ .
\end{lemma}
\begin{proof}
First, notice that $\|u_{i, n}\|^{2}_{L^{2}}$ are bounded for all $n$ and $i=1$, $2$, $3$. Moreover, from \eqref{E1} we have 
\begin{align*}
\frac{1}{2}\sum^{3}_{i=1}\|\partial_{x}u_{i,n}\|^{2}_{L^{2}}&=E(\vv{u}_{n})+\frac{1}{p+1}\sum^{3}_{i=1}\|u_{i,n}\|^{p+1}_{L^{p+1}}+\alpha\Re\int_{\mathbb{R}}u_{1,n}u_{2,n}\overline{u_{3,n}}\, dx \\
&\leq \sup_{n\in \mathbb{N}}E(\vv{u}_{n})+C\sum^{3}_{i=1}\|\partial_{x}u_{i}\|^{(p-1)/2}_{L^{2}}+C\sum^{3}_{i=1} \|\partial_{x}u_{i}\|^{\frac{1}{2}}_{L^{2}}.
\end{align*}
Since $\sum^{3}_{i=1}\|\partial_{x}u_{i,n}\|^{2}_{L^{2}}$ is now shown to be bounded by a smaller power, it follows that the sequence $\left\{\vv{u}_{n}\right\}$ is bounded in $H^{1}(\mathbb{R}; \mathbb{C}^{3})$. Thus we obtain the proof of statement i) of the lemma. 

Next we prove ii). We assume that the conclusion of statement ii) is false, then by passing to subsequence if necessary, we see that  $\lim_{n\rightarrow\infty}\|\partial_{x}u_{1, n}\|^{}_{L^{2}}=0$. Notice  that if $\left\{\vv{u}_{n}\right\}$ is a minimizing sequence, then $\left\{|\vv{u}_{n}|\right\}=\left\{(|u_{1,n}|,|u_{2,n}|, |u_{3,n}|)\right\}$ is also a minimizing sequence for problem $I(\gamma,\mu,s)$ (see Lemma \ref{DL5} below). Now since $\|\partial_{x}|u_{1, n}|\|^{2}_{L^{2}}\leq \|\partial_{x}u_{1, n}\|^{2}_{L^{2}}$, it follows that 
$\lim_{n\rightarrow\infty}\|\partial_{x}|u_{1, n}|\|^{}_{L^{2}}=0$. Moreover, from the Gagliardo-Nirenberg inequality and statement i), we have that $\lim_{n\rightarrow\infty}\|u_{1, n}\|^{p+1}_{L^{p+1}}=0$ and 
\begin{align*}
\int_{\mathbb{R}}|u_{1,n}||u_{2,n}||{u_{3,n}}|\, dx &\leq C \|u_{1,n}\|_{L^{4}}\|u_{2,n}\|_{L^{4}}\|u_{3,n}\|_{L^{2}}\\
&\leq C\|\partial_{x}|u_{1,n}|\|^{1/4}_{L^{2}}\rightarrow 0\quad \text{as $n\rightarrow +\infty$.}
\end{align*}
In particular, this implies that 
\begin{equation}\label{EE1}
I(\gamma,\mu,s )=\lim_{n\rightarrow\infty} E(|\vv{u}_{n}|)=\lim_{n\rightarrow\infty} \sum^{3}_{i=2}\left\{\frac{1}{2}\|\partial_{x}|u_{i,n}|\|^{2}_{L^{2}}-\frac{1}{p+1}\|u_{i,n}\|^{p+1}_{L^{p+1}}\right\}.
\end{equation}
On the other hand, pick any $\psi\in H^{1}(\mathbb{R})$ non-negative such that $\|\psi\|^{2}_{2}=\gamma$ and let $\psi_{\theta}(x)=\sqrt{\theta}\psi(\theta x)$. Notice that $I(\gamma,\mu,s )\leq E(\psi_{\theta}, |u_{2,n}|, |u_{3,n}|)$ for all $n$. Moreover, if we define
\begin{equation}\label{Ex1}
T=\frac{\theta^{2}}{2}\int_{\mathbb{R}}(\partial_{x}\psi)^{2}\, dx- \frac{\theta^{(p-1)/2}}{p+1}\int_{\mathbb{R}}\psi^{p+1}\, dx
\end{equation}
then $T<0$ for sufficiently small $\theta > 0$ because $(p-1)/2<2$. Thus, for all $n\in \mathbb{N}$, 
\begin{align*}
I(\gamma,\mu,s )&\leq E(\psi_{\theta}, |u_{2,n}|, |u_{3,n}|)=\sum^{3}_{i=2}\left\{\frac{1}{2}\|\partial_{x}|u_{i,n}|\|^{2}_{L^{2}}-\frac{1}{p+1}\|u_{i,n}\|^{p+1}_{L^{p+1}}\right\} \\
&-\alpha\int_{\mathbb{R}}\psi_{\theta}|u_{2,n}||{u_{3,n}}|\, dx+T\\
&\leq \sum^{3}_{i=2}\left\{\frac{1}{2}\|\partial_{x}|u_{i,n}|\|^{2}_{L^{2}}-\frac{1}{p+1}\|u_{i,n}\|^{p+1}_{L^{p+1}}\right\}+T.
\end{align*}
Thus, we have
\begin{equation*}
I(\gamma,\mu,s )\leq\lim_{n\rightarrow\infty} \sum^{3}_{i=2}\left\{\frac{1}{2}\|\partial_{x}|u_{i,n}|\|^{2}_{L^{2}}-\frac{1}{p+1}\|u_{i,n}\|^{p+1}_{L^{p+1}}\right\}+T,
\end{equation*}
a contradiction to \eqref{EE1}. Therefore, there exists a constant $\delta_{1}>0$ such that $\|\partial_{x}u_{1, n}\|^{2}_{L^{2}}\geq \|\partial_{x}|u_{1, n}|\|^{2}_{L^{2}}\geq \delta_{1}$. This completes the proof of (ii). The proof of (iii) and (iv) follow the same line of reasoning and we omit it.
\end{proof}

Now to each minimizing sequence $\left\{\vv{u}_{n}\right\}$ of problem \eqref{MP}, we associate 
the following sequence of nondecreasing functions (L\'evy concentration functions) $M_{n}:[0,\infty)\rightarrow [0,\gamma+\mu+s]$ defined by
\begin{equation*}
M_{n}(r):=\sup_{y\in \mathbb{R}}\int^{y+r}_{y-r}\left\{|u_{1, n}(x)|^{2}+|u_{2, n}(x)|^{2}+|u_{3, n}(x)|^{2}\right\}\, dx.
\end{equation*}
Since $\|u_{1,n}\|^{2}_{L^{2}}\rightarrow\gamma$,  $\|u_{2,n}\|^{2}_{L^{2}}\rightarrow\mu$ and $\|u_{3,n}\|^{2}_{L^{2}}\rightarrow s$, as $n$ goes to $+\infty$, then $\left\{M_{n}\right\}$ is a uniformly bounded sequence of nondecreasing functions on $[0,\infty)$. By Helly's selection theorem we see that $\left\{M_{n}\right\}$ must have a subsequence, which we again denote by $\left\{M_{n}\right\}$, that converges pointwise and uniformly on compact sets to a nonnegative nondecreasing function $M: [0,\infty)\rightarrow [0,\gamma+\mu+s]$. Let
\begin{equation}\label{DL}
\lambda:=\lim_{r\rightarrow\infty}\lim_{n\rightarrow\infty}\sup_{y\in \mathbb{R}}\int^{y+r}_{y-r}\left\{|u_{1, n}(x)|^{2}+|u_{2, n}(x)|^{2}+|u_{3, n}(x)|^{2}\right\}\, dx.
\end{equation}
Then, $0\leq\lambda \leq \gamma+\mu+s$.  Lions' concentration compactness lemma \cite{PLL1, PLL2} shows that there are three (mutually exclusive) possibilities for the value of $\lambda$:\\
{\rm (i)} (Vanishing)  $\lambda=0$. Since  $M(r)$ is non-negative and nondecreasing, it follows that
\begin{equation*}
M(r)=\lim_{n\rightarrow\infty} M_{n}(r)=\lim_{n\rightarrow\infty} \sup_{y\in \mathbb{R}}\int^{y+r}_{y-r}\left\{|u_{1, n}(x)|^{2}+|u_{2, n}(x)|^{2}+|u_{3, n}(x)|^{2}\right\}\, dx=0,
\end{equation*}
for every $r\in [0,\infty)$, or\\
{\rm (ii)} (Dichotomy) $\lambda\in (0, \gamma+\mu+s)$, or \\
{\rm (iii)} (Compactness) $\lambda=\gamma+\mu+s$;  in this case, there exists a sequence $y_{n}\in \mathbb{R}$ such that $|u_{1, n}(\cdot+y_{n})|^{2}+|u_{2, n}(\cdot+y_{n})|^{2}+|u_{3, n}(\cdot+y_{n})|^{2}$ is tight, namely, for all $\epsilon>0$ there exists $r(\epsilon)$ such that
\begin{equation*}
\int^{y_{n}+r(\epsilon)}_{y_{n}-r(\epsilon)}\left\{|u_{1, n}(x)|^{2}+|u_{2, n}(x)|^{2}+|u_{3, n}(x)|^{2}\right\}\, dx\geq(\gamma+\mu+s)- \epsilon.
\end{equation*}
for all sufficiently large $n$.

In what follows, we study separately the three possibilities: $\lambda=0$, $0<\lambda<\gamma+\mu+s$ and $\lambda=\gamma+\mu+s$. 
We begin by ruling out the ``vanishing'' possibility. 

The following lemma is well-know.  For a proof we refer to Lemma 3.9 in \cite{SLC}.
\begin{lemma} \label{CLLQ} Let $q>2$.  Suppose $\left\{f_{n}\right\}$ is a sequence of functions which is bounded in $H^{1}_{\mathbb{C}}(\mathbb{R})$ and for some $R>0$,
\begin{equation}\label{CLa}
\lim_{n\rightarrow\infty}\sup_{y\in\mathbb{R}^{}} \int^{y+R}_{y-R}\left|f_{n}(x)\right|^{2}dx=0.
\end{equation}
Then  $\lim_{n\rightarrow \infty}\|f_{n}\|_{L^{q}}=0$.
\end{lemma}

\begin{lemma} \label{L3}
For every minimizing sequence $\left\{\vv{u}_{n}\right\}$ for problem $I(\gamma,\mu, s)$, we have that $\lambda>0$.
\end{lemma} 
\begin{proof} Suppose on the contrary that $\lambda=0$. Then there exist a positive $R_{0}$ and a subsequence of the minimizing sequence $\left\{\vv{u}_{n}\right\}$, which we also denote by $\left\{\vv{u}_{n}\right\}$,  such that 
\begin{equation*}
\sup_{y\in\mathbb{R}^{}} \int^{y+R_{0}}_{y-R_{0}}\left\{|u_{1, n}(x)|^{2}+|u_{2, n}(x)|^{2}+|u_{3, n}(x)|^{2}\right\}\, dx\rightarrow 0, 
\end{equation*}
as $n$ goes to $+\infty$. Then \eqref{CLa} holds for $f_{n}=u_{i, n}$, for $i=1$, $2$, $3$. Since the sequence $\left\{u_{i, n}\right\}$ is bounded in $H^{1}_{\mathbb{C}}(\mathbb{R})$ by Lemma \ref{DL1}(i), then  Lemma \ref{CLLQ} says that  for all $q>2$,  $\|u_{i, n}\|_{L^{p}}\rightarrow 0$. In particular, by H\"older inequality we see that
\begin{equation*}
\left|\Re\int_{\mathbb{R}}u_{1, n}u_{2, n}\overline{u_{3, n}}\, dx\right|\leq \|u_{1, n}\|^{}_{L^{3}}\|u_{2, n}\|^{}_{L^{3}}\|u_{3, n}\|^{}_{L^{3}}\rightarrow 0.
\end{equation*}
Hence
\begin{equation*}
I(\gamma, \mu, s)=\lim_{n\rightarrow \infty}E(\vv{u}_{n})=\frac{1}{2}\sum^{3}_{i=1}\|\partial_{x}u_{i, n}\|^{2}_{L^{2}}\geq 0,
\end{equation*}
contradicting Lemma \ref{L1}.
\end{proof}

Now we rule out the possibility of dichotomy. We have divided the proof into a sequence of lemmas.
Before stating our next lemma we recall some results. 

\begin{remark}\label{AsDf}
Let $\gamma>0$. It is well known that the function
\begin{equation}\label{Ax1}
\phi_{\omega}(x)=\left\{{\omega(p+1)}\mathrm{sech}^{2}\left({((p-1)/2)\sqrt{2\omega}}\,x\right)\right\}^{\frac{1}{p-1}}
\end{equation}
minimizes the energy functional at fixed mass. More precisely, choose $\omega>0$ such that $\|\phi_{\omega}\|^{2}_{L^{2}}=\gamma$, then $\phi_{\omega}$ is a minimizer of the problem
\begin{equation}\label{Ax112}
S({\gamma}):=\inf\left\{E_{1}(u): {u}\in H^{1}(\mathbb{R}; \mathbb{C}),\,\,\|u\|^{2}_{L^{2}}=\gamma\right\},
\end{equation}
where 
\begin{equation}\label{Ebnn}
E_{1}(u)=\frac{1}{2}\|\partial_{x}u\|^{2}_{L^{2}}-\frac{1}{p+1}\|u\|^{p+1}_{L^{p+1}}.
\end{equation}
Moreover, if $\left\{f_{n}\right\}$ is any sequence of functions in $H^{1}(\mathbb{R}; \mathbb{C})$ such that $\|f_{n}\|^{2}_{L^{2}}\rightarrow\gamma$  and $E_{1}(f_{n})\rightarrow S({\gamma})$ as $n$ goes to $+\infty$, then there exist a subsequence $\left\{f_{n_{k}}\right\}$ and a sequence $\left\{y_{k}\right\}\subset \mathbb{R}$ such that $\left\{f_{n_{k}}(\cdot+y_{k})\right\}$ converges strongly in 
$H^{1}(\mathbb{R}; \mathbb{C})$ to $\phi_{\omega}(x)$. For a proof of such statement we refer to \cite{CB}.
\end{remark}

In the following lemma we show  that the value of $E(\vv{u})$ decreases when $\vv{u}$ is replaced by $|\vv{u}|$. Here and hereafter we use the notation $|\vv{u}|=(|u_{1}|, |u_{2}|, |u_{3}|)$.

\begin{lemma} \label{DL5} 
If $\vv{u}=(u_{1},u_{2},u_{3})\in H^{1}(\mathbb{R}; \mathbb{C}^{3})$, then  $E(|\vv{u}|)\leq E(\vv{u})$.
\end{lemma}
\begin{proof}
Let $f\in H_{\mathbb{C}}^{1}(\mathbb{R})$.  Since $\|\partial_{x}|f|\|^{2}_{L^{2}}\leq \|\partial_{x}f\|^{2}_{L^{2}}$ and $\alpha>0$, the lemma follows immediately.
\end{proof}

Recall that the symmetric rearrangement of a measurable function $f:[0,+\infty)\rightarrow \mathbb{R}$ is  defined by
\begin{equation*}
f^{*}(x)=\int^{+\infty}_{0}\chi^{\ast}_{\left\{|f|>t\right\}}(x)dt,
\end{equation*}
where $\chi^{\ast}_{\left\{|f|>t\right\}}$  denotes the characteristic function of a ball of volume $m(\left\{x: |f(x)|>t\right\})$  centered at the origin. It is well known that $\|f^{*}\|^{p}_{L^{p}}=\|f\|^{p}_{L^{p}}$ for any $1\leq p<\infty$ (see \cite[Chapter 3]{ELL}). Furthermore, from Lemmas 7.17 and 6.17  of \cite{ELL} we have that
\begin{equation}\label{Sdr}
\int_{\mathbb{R}}|\partial_{x}|f|^{*}|^{2}\,dx\leq\int_{\mathbb{R}}|\partial_{x}|f||^{2}\,dx\leq \int_{\mathbb{R}}|\partial_{x}f|^{2}\,dx.
\end{equation}
For the next lemma, we set $|\vv{u}|^{\ast}:=(|u_{1}|^{\ast},|u_{2}|^{\ast},|u_{3}|^{\ast})$.
\begin{lemma} \label{D456} 
Let $\vv{u}\in H^{1}(\mathbb{R}; \mathbb{C}^{3})$. Suppose that  $u_{1}$, $u_{2}$ and $u_{3}$ are functions with compact support on $\mathbb{R}$.  Then  $E(|\vv{u}|^{\ast})\leq E(|\vv{u}|)\leq E(\vv{u})$.
\end{lemma}
\begin{proof}
By Theorem 1 and Remark 2 of \cite{ISR} we have
\begin{equation}\label{DX}
\int_{\mathbb{R}}|u_{1}||u_{3}||{u_{3}}|\, dx\leq \int_{\mathbb{R}}|u_{1}|^{\ast}|u_{3}|^{\ast}|{u_{3}}|^{\ast}\, dx.
\end{equation}
Then the lemma follows  immediately from  \eqref{Sdr}, \eqref{DX} and Lemma \ref{DL5}. 
\end{proof}

The next lemma is crucial in obtaining the strict sub-additivity of the function $I(\gamma, \mu, s)$. For a proof we refer the reader to Lemma 2.10 in \cite{Albert}.
\begin{lemma} \label{Le11}
Let $f$, $g:\mathbb{R}\rightarrow [0,\infty)$ be  functions such that are non-increasing, $C^{\infty}_{c}(\mathbb{R})$ and even. Let $x_{1}$ and $x_{2}$ be numbers such that $f(x+x_{1})$ and $g(x+x_{2})$ have disjoint supports, and define $w(x):=f(x+x_{1})+g(x+x_{2})$. Then  $\partial_{x}w^{*}\in L^{2}(\mathbb{R})$ (the first derivative of $w^{*}$ in the distributional sense)  and one has the estimate
\begin{equation}\label{ESTi}
\|\partial_{x}w^{*}\|^{2}_{L^{2}}\leq \|\partial_{x}w\|^{2}_{L^{2}}-\frac{3}{4}\text{\rm min}\left\{ \|\partial_{x}f\|^{2}_{L^{2}},  \|\partial_{x}g\|^{2}_{L^{2}}\right\}.
\end{equation}
\end{lemma}

\begin{lemma} \label{cvb}
Let $\gamma$, $\mu$ , $s\geq0$.  Then there exists a minimizing sequence $\left\{\vv{u}_{n}\right\}$ of $I(\gamma, \mu, s)$ such that, for every $n$ and $i=1$, $2$, $3$, the functions $u_{i, n}\in H_{\rm rad}^{1}(\mathbb{R})\cap C^{\infty}_{c}(\mathbb{R})$, are non-negative and non-increasing for $x\geq0$. Furthermore, $\|u_{1, n}\|^{2}_{L^{2}}=\gamma$, $\|u_{2, n}\|^{2}_{L^{2}}=\mu$ and $\|u_{3, n}\|^{2}_{L^{2}}=s$ for all $n$.
\end{lemma}
\begin{proof}
We denote the convolution product of the functions $f$ and $g$ by $f\star g$. We can assume that $\gamma>0$, $\mu>0$ and $s>0$, as otherwise
just simply take $u_{i, n}$ identically zero on $\mathbb{R}$.  Start with a given minimizing sequence  $\left\{\vv{w}_{n}\right\}$ for $I(\gamma, \mu, s)$. Notice that  we can approximate $\left\{\vv{w}_{n}\right\}$ by functions $\left\{\vv{z}^{}_{n}\right\}$ which have
compact support, and  $I(\gamma, \mu,s )=\lim_{n\rightarrow\infty}E(\vv{z}^{}_{n})$. 
Now from Lemma \ref{D456}, we have that the sequence $\left\{|\vv{z}^{}_{n}|^{\ast}\right\}$ is still a minimizing sequence of $I(\gamma, \mu)$. Hence we may assume without loss of generality $\left\{\vv{w}_{n}\right\}=\left\{|\vv{z}^{}_{n}|^{\ast}\right\}$. 
Next let $\varphi\in C^{\infty}_{c}(\mathbb{R})$ be any non-negative, even, and decreasing function for $x\geq0$ satisfying $\int_{\mathbb{R}}\varphi\,dx=1$. For any arbitrary $\epsilon>0$, consider $\varphi_{\epsilon}(\cdot)=(1/\epsilon)\varphi(\cdot/\epsilon)$, and set
\begin{equation*}
(v_{1,n}, v_{2,n}, v_{3,n}):=(w_{1,n}\star \varphi_{\epsilon_{n}}, w_{2,n}\star \varphi_{\epsilon_{n}}, w_{3,n}\star \varphi_{\epsilon_{n}})
\end{equation*}
with $\epsilon_{n}$ appropriately small for $n$ large. Finally, set
\begin{equation*}
u_{1, n}:=\sqrt{\gamma}\frac{v_{1,n}}{\|v_{1,n}\|^{}_{L^{2}}}, \quad u_{2, n}:=\sqrt{\mu}\frac{v_{2,n}}{\|v_{2,n}\|^{}_{L^{2}}}, \quad
u_{3, n}:=\sqrt{s}\frac{v_{3,n}}{\|v_{3,n}\|^{}_{L^{2}}}.
\end{equation*}
Then it is not hard to show that the sequence $\left\{(u_{1,n}, u_{2,n}, u_{3,n})\right\}$ satisfies the desired properties.
\end{proof}

\begin{lemma} \label{Le1}
Let $\gamma_{1}$, $\gamma_{2}$, $\mu_{1}$, $\mu_{2}$, $s_{1}$,  $s_{2}\in [0, \infty)$, and suppose that $\gamma_{1}+\gamma_{2}>0$, $\mu_{1}+\mu_{2}>0$,  $s_{1}+s_{2}>0$,   $\gamma_{1}+\mu_{1}+s_{1}>0$ and $\gamma_{2}+\mu_{2}+s_{2}>0$. Then the function $I(\gamma, \mu, s)$ enjoys the following property
\begin{equation}\label{EII}
I(\gamma_{1}+\gamma_{2},\mu_{1}+\mu_{2}, s_{1}+s_{2})< I(\gamma_{1}, \mu_{1}, s_{1})+ I(\gamma_{2}, \mu_{2}, s_{2}).
\end{equation}
\end{lemma}
\begin{proof}
Our proof is inspired by the one of Lemma 2.12 in \cite{Albert} (see also \cite{BB2016}).  For $k = 1$, $2$, consider the minimizing sequences $\big\{(u^{(k)}_{1,n}, u^{(k)}_{2,n}, u^{(k)}_{3,n}\big\}$ of $I(\gamma_{k}, \mu_{k}, s_{k})$ as constructed in Lemma \ref{cvb}. For each $n$, let the number $x_{n}$ be such that for each, $1\leq i\leq 3$,  $u^{(1)}_{i, n}$ and $\tilde{u}^{(2)}_{i, n}(x):=u^{(2)}_{i,n}(x+x_{n})$ have disjoint support. Define:
\begin{equation*}
u_{i, n}:=\left(u^{(1)}_{i, n}+\tilde{u}^{(2)}_{i, n}\right)^{*} \quad \text{for every $i = 1$, $2$, $3$.} 
\end{equation*}
Since $\|u_{1,n}\|^{2}_{L^{2}}=\gamma_{1}+\gamma_{2}$, $\|u_{2, n}\|^{2}_{L^{2}}=\mu_{1}+\mu_{2}$  and $\|u_{3, n}\|^{2}_{L^{2}}=s_{1}+s_{2}$ we see that
\begin{equation}\label{FIneq}
I(\gamma_{1}+\gamma_{2}, \mu_{1}+\mu_{2}, s_{1}+s_{2})\leq E(u_{1, n}, u_{2, n}, u_{3, n}). 
\end{equation}
Now, applying the Lemma \ref{Le11} to each component of the sequence $\left\{(u_{1, n}, u_{2, n}, u_{3, n})\right\}$, it follows that
\begin{align}\label{CTh} 
\sum^{3}_{i=1}\|\partial_{x}u_{i, n}\|^{2}_{L^{2}}&\leq \sum^{3}_{i=1}\|\partial_{x}u^{(1)}_{i, n}\|^{2}_{L^{2}}+\sum^{3}_{i=1}\|\partial_{x}\tilde{u}^{(2)}_{i, n}\|^{2}_{L^{2}}-R_{n} 
\end{align}
where 
\begin{equation}\label{FFss}
R_{n}=\frac{3}{4}\sum^{3}_{i=1}\mbox{min}\left\{\|\partial_{x}u^{(1)}_{i, n}\|^{2}_{L^{2}}, \|\partial_{x}u^{(2)}_{i, n}\|^{2}_{L^{2}}\right\}.
\end{equation}
Moreover, from properties of rearrangement, we have
\begin{align*}
 \int_{\mathbb{R}}|u_{i,n}|^{p+1}\, dx &=\int_{\mathbb{R}}|u^{(1)}_{i,n}|^{p+1}\, dx +\int_{\mathbb{R}}|u^{(2)}_{i,n}|^{p+1}\, dx \\
\int_{\mathbb{R}}u_{1,n}u_{2,n}u_{3,n}\, dx  &\geq \int_{\mathbb{R}}u^{(1)}_{1,n}u^{(1)}_{2,n}u^{(1)}_{3,n}\, dx  +\int_{\mathbb{R}}u^{(2)}_{1,n}u^{(2)}_{2,n}u^{(2)}_{3,n}\, dx. 
\end{align*}
Thus, applying the estimate \eqref{CTh} and \eqref{FIneq} we have
\begin{align*}
I(\gamma_{1}+\gamma_{2}, \mu_{1}+\mu_{2}, s_{1}+s_{2})&\leq E(u_{1,n}, u_{2,n}, u_{3,n})\\
&\leq E(u^{(1)}_{1,n}, u^{(1)}_{2,n}, u^{(1)}_{3,n})+E(u^{(2)}_{1,n}, u^{(2)}_{2,n}, u^{(2)}_{3,n})-R_{n},
\end{align*}
hence, by taking  limit superior as $n$ goes to $+\infty$, we get
\begin{equation}\label{LSI}
I(\gamma_{1}+\gamma_{2}, \mu_{1}+\mu_{2},  s_{1}+s_{2})\leq I(\gamma_{1}, \mu_{1}, s_{1})+I(\gamma_{2},\mu_{2}, s_{2})-\liminf_{n\rightarrow\infty} R_{n}.
\end{equation}
Next we prove the strict inequality \eqref{EII}. As was observed in \cite[Lemma 2.7]{BB2016} (see also \cite{LNW}), it is sufficient to consider the following five cases: (i)  $\gamma_{1}$, $\gamma_{2}>0$ and $\mu_{1}$, $\mu_{2}$, $s_{1}$, $s_{2}\geq 0$;  (ii) $\gamma_{1}=0$, $\gamma_{2}>0$, $\mu_{2}>0$ and $s_{1}=0$; (iii) $\gamma_{1}=0$, $\gamma_{2}>0$, $\mu_{2}>0$ and $s_{1}>0$; (iv) $\gamma_{1}=0$, $\gamma_{2}>0$, $\mu_{2}=0$ and $s_{1}=0$; and (v) $\gamma_{1}=0$, $\gamma_{2}>0$, $\mu_{2}=0$ and $s_{1}>0$.
 
\textbf{Case (i)}. When $\gamma_{1}>0$ and $\gamma_{2}>0$, Lemma \ref{DL1}  guarantees that there exist a
pair of positive numbers $\kappa_{1}$ and $\kappa_{2}$ such that for a sufficiently large $n$, $\|\partial_{x}u^{(1)}_{1, n}\|^{2}_{L^{2}}\geq\kappa_{1}$ and $\|\partial_{x}u^{(2)}_{1, n}\|^{2}_{L^{2}}\geq\kappa_{2}$.
Let $\kappa = \mbox{min}(\kappa_{1}, \kappa_{2}) > 0$;  then  \eqref{FFss},  implies  that $R_{n} \geq 3\kappa/4$ for all $n$ large, and in view of \eqref{LSI}, we get
\begin{align*}
I(\gamma_{1}+\gamma_{2}, \mu_{1}+\mu_{2}, s_{1}+s_{2})&\leq I(\gamma_{1}, \mu_{1}, s_{1})+I(\gamma_{2},\mu_{2}, s_{2})-\frac{3}{4}\kappa\\
&<I(\gamma_{1}, \mu_{1}, s_{1})+I(\gamma_{2},\mu_{2}, s_{2}).
\end{align*}

\textbf{Case (ii)}. We have $\gamma_{1}=0$, $\gamma_{2}>0$, $\mu_{2}>0$ and $s_{1}=0$.  Since $\gamma_{1}+\mu_{1}+s_{1}>0$, so  $\mu_{1}>0$ as well. Once again, applying Lemma \ref{DL1} we have there exist  $\kappa_{3}>0$ and $\kappa_{4}>0$ such that for a sufficiently large $n$,   $\|\partial_{x}u^{(1)}_{2, n}\|^{2}_{L^{2}}\geq\kappa_{3}$ and $\|\partial_{x}u^{(2)}_{2, n}\|^{2}_{L^{2}}\geq\kappa_{4}$. Let $\kappa^{\ast} = \mbox{min}(\kappa_{3}, \kappa_{4}) > 0$; from \eqref{FFss}, it follows  that $R_{n} \geq 3\kappa^{\ast}/4$. Thus, by \eqref{LSI}, we obtain 
\begin{align*}
I(\gamma_{1}+\gamma_{2}, \mu_{1}+\mu_{2}, s_{1}+s_{2})&\leq I(\gamma_{1}, \mu_{1}, s_{1})+I(\gamma_{2},\mu_{2}, s_{2})-\frac{3}{4}\kappa^{\ast}\\ &<I(\gamma_{1}, \mu_{1}, s_{1})+I(\gamma_{2},\mu_{2}, s_{2}).
\end{align*}

\textbf{Case (iii)}. In this case  $\gamma_{1}=0$, $\gamma_{2}>0$, $\mu_{2}>0$ and $s_{1}>0$.   If $\mu_{1}>0$ or $s_{2}>0$, then the proofs  follow the same lines as that in the Case (ii) above. Thus, we can suppose that $\mu_{1}=0$ and $s_{2}=0$; that is, we have to prove that
\begin{align}\label{Case3}
I(\gamma_{2}, \mu_{2}, s_{1})&< I(0, 0, s_{1})+I(\gamma_{2},\mu_{2},0).
\end{align}
Now, it is clear that (see Remark \ref{AsDf})
\begin{align*}
I(0, 0, s_{1})&=\inf\left\{E_{1}(u): {u}\in H_{\mathbb{C}}^{1}(\mathbb{R}),\,\,\|u\|^{2}_{L^{2}}=s_{1}\right\},\\
I(\gamma_{2},\mu_{2},0)&=\inf\left\{E_{1}(u)+E_{1}(v): {u}, v\in H_{\mathbb{C}}^{1}(\mathbb{R}),\,\,\|u\|^{2}_{L^{2}}=\gamma_{2},\,\,\,\|v\|^{2}_{L^{2}}=\mu_{2}  \right\}.
\end{align*}
Let us consider $\omega_{1}$, $\omega_{2}$, $\omega_{3}>0$ such that $\|\phi_{\omega_{1}}\|^{2}_{L^{2}}=\gamma_{2}$,  $\|\phi_{\omega_{2}}\|^{2}_{L^{2}}=\mu_{2}$ and  $\|\phi_{\omega_{3}}\|^{2}_{L^{2}}=s_{1}$. Then, we have
\begin{align*}
I(\gamma_{2}, \mu_{2}, s_{1})&\leq E_{1}(\phi_{\omega_{1}})+E_{1}(\phi_{\omega_{2}})+ E_{1}(\phi_{\omega_{3}})-\alpha\int_{\mathbb{R}}\phi_{\omega_{1}}\phi_{\omega_{2}}\phi_{\omega_{3}}\, dx \\
&< E_{1}(\phi_{\omega_{1}})+E_{1}(\phi_{\omega_{2}})+ E_{1}(\phi_{\omega_{3}})\\
&= I(\gamma_{2},\mu_{2},0)+ I(0, 0, s_{1}).
\end{align*}
This proves the Case  (iii).

\textbf{Case (iv).} In this case  we have to prove that
\begin{align*}
I(\gamma_{2}, \mu_{1}, s_{2})&< I(0, \mu_{1}, 0)+I(\gamma_{2}, 0, s_{2}),
\end{align*}
which can be proved using similar argument as in Case (iii).

\textbf{Case (v)}. We have $\gamma_{1}=0$, $\gamma_{2}>0$, $\mu_{2}=0$ and $s_{1}>0$. If $s_{2}>0$, then the proof follows the same argument as that in the Case (ii) above. Thus, we may assume that $s_{2}=0$, and prove that
\begin{align*}
I(\gamma_{2}, \mu_{1}, s_{1})&< I(0, \mu_{1}, s_{1})+I(\gamma_{2}, 0, 0).
\end{align*}
But then the proof is very similar to the proof of \eqref{Case3}. This completes the proof of \eqref{EII} in all cases.
\end{proof}

\begin{lemma} \label{LF3}
Let $\lambda$ be as defined in \eqref{DL}. Suppose $\gamma$, $\mu$ $s>0$ and let $\left\{\vv{u}_{n}\right\}$ be any minimizing sequence for
$I(\gamma_{}, \mu_{}, s)$.  Then there exist $\gamma_{1}\in [0, \gamma]$,  $\mu_{1}\in [0, \mu]$ and $s_{1}\in [0, s]$ such that $\lambda=\gamma_{1}+\mu_{1}+s_{1}$ and 
\begin{equation}\label{INES}
I_{}(\gamma_{1}, \mu_{1}, s_{1})+I_{}(\gamma-\gamma_{1}, \mu-\mu_{1}, s-s_{1})\leq I_{}(\gamma, \mu, s).
\end{equation}
\end{lemma}
\begin{proof} To show this we will follow the arguments in \cite{Albert} and \cite{BB2016}. Let $\sigma\in C^{\infty}_{0}[-2,2]$ be  such that $\sigma\equiv1$ on $[-1,1]$, and let $\rho\in C^{\infty}(\mathbb{R})$ be such that $\sigma^{2}+\rho^{2}=1$ on $\mathbb{R}$. Set, for  $r>0$, the rescaling: $\sigma_{r}(x)=\sigma(x/r)$ and  $\rho_{r}(x)=\rho(x/r)$. Given $\epsilon$ an arbitrary positive number, for all sufficiently large $r$ we have  $\lambda-\epsilon<M(r)\leq M(2r)\leq\lambda$. By taking $r$ larger if necessary we may assume that $1/r<\epsilon$. Thus, by definition of $M$, we may choose $N\in \mathbb{N}$ large enough so that
\begin{equation*}
\lambda-\epsilon<M_{n}(r)\leq M_{n}(2r)\leq \lambda+\epsilon,
\end{equation*}
for all $n\geq N$. Consequently, for each $n\geq N$, we can find $y_{n}$ such that
\begin{equation}\label{I1}
\int^{y_{n}+r}_{y_{n}-r}\sum^{3}_{i=1}|u_{i, n}|^{2}\,dx>\lambda-\epsilon\quad \text{and}\quad\int^{y_{n}+2r}_{y_{n}-2r}\sum^{3}_{i=1}|u_{i, n}|^{2}\,dx<\lambda+\epsilon.
\end{equation}
Next we define the sequences
\begin{align*}
&(l_{1, n}(x), l_{2, n}(x), l_{3, n}(x)):=(\sigma_{r}(x-y_{n})u_{1, n}(x), \sigma_{r}(x-y_{n})u_{2, n}(x), \sigma_{r}(x-y_{n})u_{3, n}(x)),\\
&(k_{1, n}(x), k_{2, n}(x), k_{3, n}(x)):=(\rho_{r}(x-y_{n})u_{1, n}(x), \rho_{r}(x-y_{n})u_{2, n}(x), \rho_{r}(x-y_{n})u_{3, n}(x)).
\end{align*}
From Lemma \ref{L1}, it follows that the sequences $\left\{l_{i, n}\right\}$ and $\left\{k_{i, n}\right\}$ are bounded in $L^{2}(\mathbb{R})$. 
Thus, possibly for a subsequence only, we see that $\|l_{1, n}\|^{2}_{L^{2}}\rightarrow \gamma_{1}$,  $\|l_{2, n}\|^{2}_{L^{2}}\rightarrow \mu_{1}$  and $\|l_{3, n}\|^{2}_{L^{2}}\rightarrow s_{1}$, where $\gamma_{1}\in [0, \gamma]$, $\mu_{1}\in [0, \mu]$ and $s_{1}\in [0, s]$. Furthermore, we also have  $\|k_{1, n}\|^{2}_{L^{2}}\rightarrow \gamma-\gamma_{1}$,  $\|k_{2, n}\|^{2}_{L^{2}}\rightarrow \mu-\mu_{1}$  and $\|k_{3, n}\|^{2}_{L^{2}}\rightarrow s-s_{1}$, and 
\begin{equation}\label{vbm}
\gamma_{1}+\mu_{1}+s_{1}=\lim_{n\rightarrow\infty}\int_{\mathbb{R}}\sum^{3}_{i=1}|u_{i, n}|^{2}\,dx=\lim_{n\rightarrow\infty}\int_{\mathbb{R}}\sum^{3}_{i=1}\sigma^{2}_{r}|u_{i, n}|^{2}\,dx.
\end{equation}
For ease of notation, here and hereafter we write $\sigma_{r}$  instead of $\sigma_{r}(x-y_{n})$, and $\rho_{r}$  instead of $\rho_{r}(x-y_{n})$. From \eqref{I1} and \eqref{vbm}, we obtain 
\begin{equation*}
|(\gamma_{1}+\mu_{1}+s_{1})-\lambda|<\epsilon.
\end{equation*}
We claim now that, for every $n$,
\begin{equation}\label{EStm}
E(l_{1, n}, l_{2, n}, l_{3, n})+E(k_{1, n}, k_{2, n}, k_{3, n})\leq E(u_{1, n}, u_{2, n}, u_{3, n})+C\epsilon.
\end{equation}
Indeed, first notice that
\begin{align*}
&E(l_{1, n}, l_{2, n}, l_{3, n})=\int_{\mathbb{R}}\sigma^{2}_{r}\left\{\sum^{3}_{i=1}\left(\frac{1}{2}|\partial_{x}u_{i, n}|^{2}-\frac{1}{p+1}|u_{i, n}|^{p+1}\right)-\alpha\Re(u_{1, n}u_{2, n}\overline{u_{3, n}})\right\}\,dx\\
&+ \int_{\mathbb{R}}\frac{(\sigma^{2}_{r}-\sigma^{p+1}_{r})}{p+1}\sum^{3}_{i=1}|u_{i, n}|^{p+1}\,dx-\alpha\Re\int_{\mathbb{R}}(\sigma^{2}_{r}-\sigma^{3}_{r})u_{1, n}u_{2, n}\overline{u_{3, n}}\,dx\\
&+\int_{\mathbb{R}}\sum^{3}_{i=1}(\partial_{x}\sigma_{r})^{2}|u_{i, n}|^{2}\,dx+2\int_{\mathbb{R}}\sum^{3}_{i=1}\sigma_{r}\partial_{x}\sigma_{r}
\Re(u_{i, n}\overline{\partial_{x}u_{i, n}})\,dx\\
&\leq\int_{\mathbb{R}}\sigma^{2}_{r}\left\{\sum^{3}_{i=1}\left(\frac{1}{2}|\partial_{x}u_{i, n}|^{2}-\frac{1}{p+1}|u_{i, n}|^{p+1}\right)-\alpha\Re(u_{1, n}u_{2, n}\overline{u_{3, n}})\right\}\,dx+C\epsilon,
\end{align*}
because of the following inequalities:\\
(i) $\|\partial_{x}\sigma_{r}\|_{L^{\infty}}=\|\partial_{x}\sigma_{}\|_{L^{\infty}}/r\leq C\epsilon$;\\
(ii) Using \eqref{I1}, one can see that
\begin{equation*}
\left|\int_{\mathbb{R}}\frac{(\sigma^{2}_{r}-\sigma^{p+1}_{r})}{p+1}\sum^{3}_{i=1}|u_{i, n}|^{p+1}\,dx\right|\leq C \sum^{3}_{i=1}\int_{r\leq|x-y_{n}|\leq 2r}|u_{i, n}|^{2}\leq C\epsilon;
\end{equation*}
(iii) Again using \eqref{I1}, we have
\begin{equation*}
\left|\alpha\Re\int_{\mathbb{R}}(\sigma^{2}_{r}-\sigma^{3}_{r})u_{1, n}u_{2, n}\overline{u_{3, n}}\,dx\right|\leq C \|u_{3, n}\|_{L^{\infty}} \sum^{2}_{i=1}\int_{r\leq|x-y_{n}|\leq 2r}|u_{i, n}|^{2}\leq C\epsilon.
\end{equation*}
Similarly, we get
\begin{align*}
E(k_{1, n}, k_{2, n}, k_{3, n})&\leq\int_{\mathbb{R}}\rho^{2}_{r}\left\{\sum^{3}_{i=1}\left(\frac{1}{2}|\partial_{x}u_{i, n}|^{2}-\frac{1}{p+1}|u_{i, n}|^{p+1}\right)\right\}\, dx\\
&-\alpha\Re\int_{\mathbb{R}}\rho^{2}_{r}(u_{1, n}u_{2, n}\overline{u_{3, n}})\,dx+C\epsilon.
\end{align*}
Since $\sigma_{r}^{2}+\rho_{r}^{2}\equiv 1$ on $\mathbb{R}$, we have \eqref{EStm}. 

Next, suppose  $\gamma_{1}$, $\mu_{1}$, $s_{1}$, $\gamma_{}-\gamma_{1}$, $\mu_{}-\mu_{1}$ and $s_{}-s_{1}$ are all positive. We set $a_{n, 1}=\sqrt{\gamma_{1}}/\|l_{1, n}\|^{}_{L^{2}}$, $a_{n, 2}=\sqrt{\mu_{1}}/\|l_{2, n}\|^{}_{L^{2}}$ and $a_{n, 3}=\sqrt{s_{1}}/\|l_{3, n}\|^{}_{L^{2}}$; and $b_{n, 1}=\sqrt{\gamma-\gamma_{1}}/\|k_{1, n}\|^{}_{L^{2}}$, $b_{n, 2}=\sqrt{\mu-\mu_{1}}/\|k_{2, n}\|^{}_{L^{2}}$ and $b_{n, 3}=\sqrt{s-s_{1}}/\|k_{3, n}\|^{}_{L^{2}}$. Then we see that  $\|a_{1, n}l_{1, n}\|^{2}_{L^{2}}=\gamma_{1}$, $\|a_{2, n}l_{2, n}\|^{2}_{L^{2}}=\mu_{1}$, $\|a_{3, n}l_{3, n}\|^{2}_{L^{2}}=s_{1}$, $\|b_{1, n}k_{1, n}\|^{2}_{L^{2}}=\gamma-\gamma_{1}$, $\|b_{2, n}k_{2, n}\|^{2}_{L^{2}}=\mu-\mu_{1}$, and  $\|b_{3, n}k_{3, n}\|^{2}_{L^{2}}=s-s_{1}$. As all the scaling factors tend to $1$ as $n$ goes to $+\infty$, we obtain
\begin{align*}
\liminf_{n\rightarrow\infty}\left\{E(l_{1, n}, l_{2, n}, l_{3, n})+E(k_{1, n}, k_{2, n}, k_{3, n})\right\}\geq I_{}(\gamma_{1}, \mu_{1}, s_{1})+I(\gamma-\gamma_{1}, \mu-\mu_{1}, s-s_{1}).
\end{align*}
Now if $\gamma_{1}=0$, $\mu_{1}>0$ and $s_{1}>0$, then $\|l_{1, n}\|^{2}_{L^{2}}\rightarrow 0$; this imples
\begin{align*}
\Re\int_{\mathbb{R}}l_{1, n}l_{2, n}\overline{l_{3, n}}\, dx\rightarrow 0, \,\,\, \text{and}\,\,\, \int_{\mathbb{R}}|l_{1, n}|^{p+1}\, dx\rightarrow 0,
\end{align*}
as $n$ goes to $+\infty$. Therefore (see Remark \ref{AsDf}),
\begin{align*}
\lim_{n\rightarrow\infty}E(l_{1, n}, l_{2, n}, l_{3, n})&=\lim_{n\rightarrow\infty}\left\{E_{1}(l_{1, n})+E_{1}(l_{2, n})+E_{1}(l_{3, n})-\alpha\int_{\mathbb{R}}l_{1, n}l_{2, n}l_{3, n}\, dx \right\}\\
&\geq\liminf_{n\rightarrow\infty}\left\{E_{1}(l_{2, n})+E_{1}(l_{3, n})\right\}\geq I_{}(0, \mu_{1}, s_{1}).
\end{align*}
Similar estimates hold if $\mu_{1}$ or $s_{1}$ are zero. Moreover, if $\gamma_{1}=0$, $\mu_{1}=0$ and $s_{1}>0$, then it easily follows that
$\lim_{n\rightarrow\infty}E(l_{1, n}, l_{2, n}, l_{3, n})\geq I_{}(0, 0, s_{1})$. Similar estimates hold with $\gamma-\gamma_{1}$ $\mu-\mu_{1}$, and $s-s_{1}$ playing the roles of $\gamma_{1}$ $\mu_{1}$, and $s_{1}$, respectively. Therefore, in all the cases we have that the
limit inferior as $n$ goes to $+\infty$ of the left hand side of \eqref{EStm}$\geq  I_{}(\gamma_{1}, \mu_{1}, s_{1})+I_{}(\gamma-\gamma_{1}, \mu-\mu_{1}, s-s_{1})$. Finally, we  take the limit inferior of the left-hand side of \eqref{EStm} as $n$ goes to $+\infty$,  and the limit of the right-hand side of \eqref{EStm} to obtain
\begin{equation*}
I_{}(\gamma_{1}, \mu_{1}, s_{1})+I_{}(\gamma-\gamma_{1}, \mu-\mu_{1}, s-s_{1})\leq I_{}(\gamma,\mu, s)+C\epsilon;
\end{equation*}
that is, $I_{}(\gamma_{1}, \mu_{1}, s_{1})+I_{}(\gamma-\gamma_{1}, \mu-\mu_{1}, s-s_{1})\leq I_{}(\gamma, \mu, s)$, as $\epsilon$ is arbitrary.
This completes the proof of the lemma.
\end{proof}

The following result rules out the possibility of dichotomy of minimizing sequences. 
\begin{lemma} \label{Ctwo}
Let $\gamma>0$, $\mu>0$ and $s>0$. Then for every minimizing  sequence of $I(\gamma, \mu, s)$, we have  $\lambda\notin(0,\gamma+\mu+s)$; that is, the case of dichotomy cannot occur.
\end{lemma}
\begin{proof}
We prove this by contradiction. Suppose that dichotomy happens; that is, $0<\lambda<\gamma+\mu+s$. Let $\gamma_{1}$, $\mu_{1}$ and $s_{1}$ be defined as in Lemma \ref{LF3}, and set $\gamma_{2}=\gamma-\gamma_{1}$, $\mu_{2}=\mu-\mu_{1}$ and $s_{2}=s-s_{1}$. It follows that $\gamma_{2}+\mu_{2}+s_{2}=(\gamma+\mu+s)-\lambda>0$ and $\gamma_{1}+\mu_{1}+s_{1}=\lambda>0$. Since $\gamma_{1}+\gamma_{2}=\gamma>0$, $\mu_{1}+\mu_{2}=\mu>0$ and $s_{1}+s_{2}=s>0$, it follows from Lemma \ref{Le1},
\begin{equation*}
I(\gamma_{1}+\gamma_{2},\mu_{1}+\mu_{2}, s_{1}+s_{2})< I(\gamma_{1}, \mu_{1}, s_{1})+ I(\gamma_{2}, \mu_{2}, s_{2}),
\end{equation*}
which is a contradiction with inequality \eqref{INES}. This completes the proof.
\end{proof}

By Lemmas \ref{L3} and \ref{Ctwo} we see that $\lambda=\gamma+\mu+s$; that is, every minimizing sequence for $I(\gamma, \mu, s)$ must be compact.  

\begin{lemma}\label{L20}
Suppose $\gamma$, $\mu>0$ and $s>0$. Let $\left\{\vv{u}_{n}\right\}$ be any minimizing sequence for $I(\gamma, \mu, s)$. Then there exists a sequence of real numbers $\left\{y_{n}\right\}$ such that the sequence $\left\{\vv{u}_{n}(\cdot+y_{n})\right\}$ has a subsequence which converges strongly in $H^{1}(\mathbb{R}, \mathbb{C}^{3})$ to some $\vv{u}$, which is a minimizer for $I(\gamma, \mu, s)$.
\end{lemma}
\begin{proof}
In this proof we will often extract subsequences without explicitly mentioning this fact. Since $\lambda=\gamma+\mu+s$,  by Lions' concentration compactness lemma \cite{PLL1} we see that there exists $\delta_{k}\in \mathbb{R}$ such that, for every $k\in \mathbb{N}$, one has 
\begin{equation*}
\int^{\delta_{k}}_{-\delta_{k}}\left\{|w_{1,n}(x)|^{2}+|w_{2,n}(x)|^{2}+|w_{3,n}(x)|^{2}\right\}\, dx>(\gamma+\mu+s)-\frac{1}{k},
\end{equation*}
for all sufficiently large $n$. For ease of notation, here and hereafter we write $w_{i, n}(x)=u_{i, n}(x+y_{n})$ for $i=1$, $2$, $3$. Thus due to the compactness of the embedding $H_{\rm loc}^{1}(\mathbb{R})\hookrightarrow L_{\rm loc}^{2}(\mathbb{R})$, it follows  that  $\big\{\vv{w}_{n}\big\}$ converges in $L^{2}[-\delta_{k},\delta_{k}]$-norm to a  limit function  $\vv{u}=(u_{1}, u_{2}, u_{3})$ satisfying
\begin{equation*}
\int^{\delta_{k}}_{-\delta_{k}}\left\{|u_{1}(x)|^{2}+|u_{2}(x)|^{2}+|u_{3}(x)|^{2}\right\}\, dx>(\gamma+\mu+s)-\frac{1}{k}.
\end{equation*}
Applying Cantor's diagonalization process, together with the fact that
\begin{equation*}
\int_{\mathbb{R}}\left\{|w_{1,n}(x)|^{2}+|w_{2,n}(x)|^{2}+|w_{3,n}(x)|^{2}\right\}\, dx=\gamma+\mu+s,
\end{equation*}
we obtain that some subsequence of $\big\{(w_{1, n}, w_{2, n}, w_{3, n})\big\}$ converges in $L^{2}(\mathbb{R})$-norm to a limit function  $\vv{u}\in (L^{2}(\mathbb{R}))^{3}$ satisfying
\begin{equation*}
\int_{\mathbb{R}}\left\{|u_{1}(x)|^{2}+|u_{2}(x)|^{2}+|u_{3}(x)|^{2}\right\}\, dx=\gamma+\mu+s.
\end{equation*}
Moreover, it is clear that $\vv{w}_{n}\rightharpoonup \vv{u}$  weakly in $H^{1}(\mathbb{R}, \mathbb{C}^{3})$ and $\vv{w}_{n}\rightarrow\vv{u}$ in $L^{2}(\mathbb{R})$-norm. Next, from the Gagliardo-Nirenberg inequality and H\"older inequality, we get
\begin{equation*}
\|w_{i,n}-u_{i}\|^{p+1}_{L^{p+1}}\leq C\|w_{i,n}-u_{i}\|^{(p+3)/2}_{L^{2}}\rightarrow 0, \quad \Re\int_{\mathbb{R}}w_{1,n}w_{2,n}\overline{w_{3,n}}\,dx\rightarrow \Re\int_{\mathbb{R}}u_{1}u_{2}\overline{u_{3}}\,dx,
\end{equation*}
as $n$ goes to $+\infty$. Consequently, it follows that
\begin{equation*}
E(\vv{u})\leq \lim_{n\rightarrow\infty}E(w_{1,n}, w_{2,n}, w_{3,n})=I(\gamma, \mu, s). 
\end{equation*}
Therefore $E(\vv{u})=I(\gamma, \mu, s)$ and $\vv{u}\in \mathcal{G}_{\gamma, \mu,s}$. Finally, since $E(\vv{u})\rightarrow E(\vv{w}_{n})$, $\|w_{i,n}\|^{p+1}_{L^{p+1}}\rightarrow \|u_{i}\|^{p+1}_{L^{p+1}}$, $\|w_{i,n}\|^{2}_{L^{2}}\rightarrow \|u_{i}\|^{2}_{L^{2}}$ and $\Re\int_{\mathbb{R}}w_{1,n}w_{2,n}\overline{w_{3,n}}\,dx \rightarrow \Re\int_{\mathbb{R}}u_{1}u_{2}\overline{u_{3}}\,dx$, we obtain 
\begin{equation*}
\|\vv{u}\|^{2}_{H^{1}(\mathbb{R}, \mathbb{C}^{3})}=\lim_{n\rightarrow\infty}\|\vv{w}_{n}\|^{2}_{H^{1}(\mathbb{R}, \mathbb{C}^{3})}.
\end{equation*}
Thus, since $\vv{w}_{n}\rightharpoonup \vv{u}$  weakly in $H^{1}(\mathbb{R}, \mathbb{C}^{3})$, we see that $\vv{w}_{n}$ converges strongly to $\vv{u}$ in $H^{1}(\mathbb{R}, \mathbb{C}^{3})$, and hence the result follows.
\end{proof}

\begin{proof}[\bf {Proof of Theorem \ref{SPI}}] The statement (i) follows immediately from Lemma \ref{L20}.  Next we prove statement (ii) of  the theorem. Suppose that \eqref{ABCD} is false. Then there exist a subsequence $\left\{\vv{u}_{n_{k}}\right\}$ of 
$\left\{\vv{u}_{n}\right\}$ and a number $\epsilon>0$ such that,
\begin{equation*}
\lim_{n\rightarrow\infty}\inf_{\vv{g}\in\mathcal{G}_{\gamma, \mu,s}, y\in \mathbb{R}}\|\vv{u}_{n_{k}}(\cdot+y)-\vv{g}\|_{H^{1}(\mathbb{R}; \mathbb{C}^{3})}\geq \epsilon.
\end{equation*}
As $\left\{\vv{u}_{n_{k}}\right\}$ is itself a minimizing sequence for $I(\gamma, \mu, s)$, we see from (i) that there exist a subsequence of real numbers $\left\{y_{k}\right\}$ and $\vv{w}\in \mathcal{G}_{\gamma, \mu,s}$ such that
\begin{equation*}
\liminf_{n\rightarrow\infty}\|\vv{u}_{n_{k}}(\cdot+y_{k})-\vv{w}\|_{H^{1}(\mathbb{R}; \mathbb{C}^{3})}=0.
\end{equation*}
This contradiction proves \eqref{ABCD}. Finally, since $\vv{w}(\cdot-y_{k})\in \mathcal{G}_{\gamma, \mu,s}$, the statement (iii) follows from statement (ii).  
On the other hand, if $\vv{u}\in\mathcal{G}_{\gamma, \mu,s}$, then by the Lagrange multiplier principle there exist numbers $\omega_{1}$, $\omega_{2}$ and $\omega_{3}$ such that 
\begin{equation*}
E^{\prime}(u_{1}, u_{2}, u_{3})=2\omega_{1}u_{1}+2\omega_{2}u_{2}+2\omega_{3}u_{3},
\end{equation*}
where the prime denotes the Fr\'echet derivative. Therefore, by computing the Fr\'echet derivative, we see that equation \eqref{SP} holds at least in the sense of distributions. Moreover, from \cite[Lemma 1.3]{Tao}, we have that $\vv{u}$ is in fact smooth and  classical solution of this equation.  Next we can write $u_{j}(x)=e^{i\theta_{j}(x)}\rho_{j}(x)$, where $\theta_{j}$, $\rho_{j}\in C^{2}(\mathbb{R})$ and $\rho_{j}\geq 0$ for $j=1$, $2$, $3$. In addition, 
\begin{align}\label{Dies}
\|\partial_{x}u_{j}\|^{2}_{L^{2}}=\int_{\mathbb{R}}|\theta_{j}^{\prime}(x)|^{2}\rho^{2}_{j}(x)\,dx +\|\partial_{x}|u_{j}|\|^{2}_{L^{2}}, \quad \text{for $j=1$, $2$, $3$.}
\end{align}
From Lemma \ref{DL5} we have $E(|u_{1}|, |u_{2}|, |u_{3}|)=E(u_{1}, u_{2}, u_{3})$, this implies that   
\begin{align*}
\sum^{3}_{i=1}\|\partial_{x}|u_{i}|\|^{2}_{L^{2}}-\alpha\,\int_{\mathbb{R}}|u_{1}||u_{3}||{u_{3}}|\, dx=\sum^{3}_{i=1}\|\partial_{x}u_{i}\|^{2}_{L^{2}}-\alpha\,\Re\int_{\mathbb{R}}u_{1}u_{3}\overline{u_{3}}\, dx.
\end{align*}
Since $\|\partial_{x}|u_{j}|\|^{2}_{L^{2}}\leq \|\partial_{x}u_{j}\|^{2}_{L^{2}}$ and $\Re\int_{\mathbb{R}}u_{1}u_{3}\overline{u_{3}}\, dx\leq \int_{\mathbb{R}}|u_{1}||u_{3}||{u_{3}}|\, dx$, it follows that $\|\partial_{x}|u_{j}|\|^{2}_{L^{2}}= \|\partial_{x}u_{j}\|^{2}_{L^{2}}$ for every $j=1$, $2$, $3$, and 
\begin{equation}\label{RIh}
\Re\int_{\mathbb{R}}u_{1}u_{3}\overline{u_{3}}\, dx= \int_{\mathbb{R}}|u_{1}||u_{3}||{u_{3}}|\, dx.
\end{equation}
Therefore, since $\rho_{j}$, $\theta_{j}\in C^{2}(\mathbb{R})$, from \eqref{Dies} it follows that $|\theta_{j}^{\prime}(x)|^{2}\rho^{2}_{j}(x)=0$ for all  $x\in\mathbb{R}$ ; this implies that  
$\theta_{j}(x)\equiv\text{constant}=\theta_{j}$. Thus, $u_{j}(x)=e^{i\theta_{j}}\rho_{j}(x)$  on $\mathbb{R}$. Notice that $\rho_{j}\geq 0$ and $\rho_{j}\neq 0$. 

Now we claim that 
\begin{equation}\label{iuy}
\int_{\mathbb{R}}|u_{1}||u_{3}||{u_{3}}|\, dx=\int_{\mathbb{R}}\rho_{1}\rho_{3}{\rho_{3}}\, dx>0.
\end{equation}
Indeed, suppose that $\int_{\mathbb{R}}\rho_{1}\rho_{3}{\rho_{3}}\, dx=0$. Since $I(\gamma, \mu, s)=E(\rho_{1}, \rho_{2}, \rho_{3})$,  we see that $I(\gamma, \mu, s)=E_{1}(\rho_{1})+E_{1}(\rho_{2})+E_{1}(\rho_{3})$, where $E_{1}$ is defined in Remark \ref{AsDf}. On the other hand, let us consider $\omega_{1}$, $\omega_{2}$, $\omega_{3}>0$ such that $\|\phi_{\omega_{1}}\|^{2}_{L^{2}}=\gamma$,  $\|\phi_{\omega_{2}}\|^{2}_{L^{2}}=\mu$ and $\|\phi_{\omega_{3}}\|^{2}_{L^{2}}=s$. Notice that $\int_{\mathbb{R}}\phi_{\omega_{1}}\phi_{\omega_{2}}\phi_{\omega_{3}}dx>0$. Then,  from  Remark \ref{AsDf} it is clear that
\begin{align*}
I(\gamma, \mu, s)&\leq E(\phi_{\omega_{1}}, \phi_{\omega_{2}}, \phi_{\omega_{3}})\\
&<E_{1}(\phi_{\omega_{1}})+E_{1}(\phi_{\omega_{2}})+E_{1}(\phi_{\omega_{3}})\\
&\leq E_{1}(\rho_{1})+E_{1}(\rho_{2})+E_{1}(\rho_{3}),
\end{align*}
which is a contradiction. This shows our claim. Finally, by \eqref{RIh} and \eqref{iuy} we obtain that $\Re (e^{i(\theta_{1}+\theta_{2}-\theta_{3})})=1$, since $|e^{i(\theta_{1}+\theta_{2}-\theta_{3})}|=1$, it follows that $e^{i(\theta_{1}+\theta_{2})}=e^{i\theta_{3}}$. This completes the proof of the theorem.
\end{proof}

\section{Proof of Theorem \ref{CSW}}
\label{S:3/2}
In this section, we prove Theorem \ref{CSW}. In the proofs in this section we follow some ideas in \cite{JLB23}.  Before stating our next lemma we recall a result from \cite{CB}. We define on $H_{\mathbb{C}}^{1}(\mathbb{R})$ the following functional
\begin{equation*}
E_{2}(u)=\frac{1}{2}\|\partial_{x}u\|^{2}_{L^{2}}-\frac{\alpha+\beta}{3}\|u\|^{3}_{L^{3}},
\end{equation*}
and the real number $J({\gamma})$ by
\begin{equation*}
J({\gamma}):=\inf\left\{E_{2}(u): {u}\in H_{\mathbb{C}}^{1}(\mathbb{R}),\,\,\|u\|^{2}_{L^{2}}=\gamma\right\}.
\end{equation*}
Then 
\begin{equation}\label{Nvs}
\psi_{\omega}(x)=\frac{3\omega}{(\alpha+\beta)}\mathrm{sech}^{2}\left(\frac{1}{2}\sqrt{2\omega}\,x\right)
\end{equation}
minimizes the energy functional at fixed mass. More precisely, if $\gamma(\omega)={12\sqrt{2}\omega^{3/2}}/{(\alpha+\beta)^{2}}$, then
$\|\psi_{\omega}\|^{2}_{L^{2}}=\gamma(\omega)$,  and $E_{2}(\psi_{\omega})=J({\gamma(\omega)})$.

\begin{lemma} \label{Lolor}
The following properties hold.\\
{\rm (i)} For every $\gamma>0$, we have $I({\gamma, \gamma, \gamma})=3J({\gamma})$. \\
{\rm (ii)} If $(f_{1}, f_{2}, f_{3})\in \mathcal{G}_{\gamma, \gamma, \gamma}$, then there exist numbers $\theta_{j}\in \mathbb{R}$ and  a non-negative real function $\rho$ such that $f_{1}(x)=e^{i\theta_{1}}\rho(x)$, $f_{2}(x)=e^{i\theta_{2}}\rho(x)$ and $f_{3}(x)=e^{i(\theta_{1}+\theta_{2})}\rho(x)$  for all $x\in \mathbb{R}$ and $j=1$, $2$, $3$.
\end{lemma}
\begin{proof}
First, by applying H\"older inequality and Young's inequality we obtain
\begin{align}\nonumber
\Re\int_{\mathbb{R}}f_{1}f_{2}\overline{f_{3}}\, dx&\leq \int_{\mathbb{R}}|f_{1}||f_{2}||{f_{3}}|\, dx\leq \|f_{1}\|^{3}_{L^{3}}\|f_{2}\|^{3}_{L^{3}}\|f_{3}\|^{3}_{L^{3}}\\ \label{CVNa}
&\leq \frac{1}{3} \|f_{1}\|^{3}_{L^{3}}+\frac{1}{3} \|f_{2}\|^{3}_{L^{3}}+\frac{1}{3}\|f_{3}\|^{3}_{L^{3}}.
\end{align}
This implies that
\begin{align*}
E(\vv{f})&=\sum^{3}_{i=1}\left\{\frac{1}{2}\int_{\mathbb{R}}|\partial_{x}f_{i}|^{2}\,dx-\frac{\beta}{3}\int_{\mathbb{R}}|f_{i}|^{3}\,dx\right\}-\alpha\,\Re\int_{\mathbb{R}}f_{1}f_{2}\overline{f_{3}}\, dx\\
&\geq\sum^{3}_{i=1}\left\{\frac{1}{2}\int_{\mathbb{R}}|\partial_{x}f_{i}|^{2}\,dx-\left(\frac{\beta+\alpha}{3}\right)\int_{\mathbb{R}}|f_{i}|^{3}\,dx\right\}\\
&=E_{2}(f_{1})+E_{2}(f_{2})+E_{2}(f_{3}),
\end{align*}
where $\vv{f}=(f_{1}, f_{2}, f_{3})$. Therefore, taking the infima on both sides of the above inequity,  we see that $I({\gamma, \gamma, \gamma})\geq 3J({\gamma})$. Now, since $\|\phi_{\omega}\|^{2}_{L^{2}}=\gamma$ for any $\omega>0$, we infer that 
\begin{align*}
I({\gamma, \gamma, \gamma})&\leq E(\phi_{\omega}, \phi_{\omega}, \phi_{\omega})=E_{2}(\phi_{\omega})+E_{2}(\phi_{\omega})+E_{2}(\phi_{\omega})=3J({\gamma}).
\end{align*}
Hence $I({\gamma, \gamma, \gamma})=3J({\gamma})$. Thus we obtain the proof of statement (i) of the lemma.
Next we prove (ii).  Let $(f_{1}, f_{2}, f_{3})\in \mathcal{G}_{\gamma, \gamma, \gamma}$. From \eqref{CVNa} and statement (i) we have
\begin{align*}
I({\gamma, \gamma, \gamma})=E(\vv{f})\geq E_{2}(f_{1})+E_{2}(f_{2})+E_{2}(f_{3})\geq 3J({\gamma})=I({\gamma, \gamma, \gamma}).
\end{align*}
Thus, $E(\vv{f})=E_{2}(f_{1})+E_{2}(f_{2})+E_{2}(f_{3})$. In particular, it follows that
\begin{align}\label{IEp}
\frac{1}{3}\sum^{3}_{i=1}\|f_{i}\|^{3}_{L^{3}}=\Re\int_{\mathbb{R}}f_{1}f_{2}\overline{f_{3}}\, dx\leq \int_{\mathbb{R}}|f_{1}||f_{2}|||{f_{3}}|\, dx\leq \frac{1}{3}\sum^{3}_{i=1}\|f_{i}\|^{3}_{L^{3}}.
\end{align}
From Theorem \ref{SPI} we may write $f_{j}(x)=e^{\theta_{j}}\rho_{j}(x)$, where $\theta_{j}\in \mathbb{R}$ and $\rho_{j}\in C^{2}(\mathbb{R})$, and $\rho_{j}\geq 0$ for $j=1$, $2$, $3$. Thus, by \eqref{IEp} we get  
\begin{align*}
\int_{\mathbb{R}}\rho_{1}(x)\rho_{2}(x){\rho_{3}(x)}\, dx=\|\rho_{1}\|^{3}_{L^{3}}\|\rho_{2}\|^{3}_{L^{3}}\|\rho_{3}\|^{3}_{L^{3}};
\end{align*}
that is, these functions $\rho_{j}$ satisfy the equality in H\"older's inequality. Therefore, we see that $\rho_{1}(x)/ \|\rho_{1}\|^{3}_{L^{3}}=\rho_{2}(x)/ \|\rho_{2}\|^{3}_{L^{3}}=\rho_{3}(x)/ \|\rho_{3}\|^{3}_{L^{3}}$. Since $\|\rho_{j}\|^{2}_{L^{2}}=\gamma$, it follows that $\rho_{1}(x)=\rho_{2}(x)=\rho_{3}(x)$ on $\mathbb{R}$. Hence the lemma is proved.
\end{proof}

Now we give the proof of Theorem \ref{CSW}.
\begin{proof}[\bf {Proof of Theorem \ref{CSW}}] 
Notice first that from Lemma \ref{Lolor}(i),  it follows that 
\begin{equation*}
\left\{\left(e^{i\theta_{1}}\psi_{\omega}(\cdot+y), e^{i\theta_{2}}\psi_{\omega}(\cdot+y), e^{i(\theta_{1}+\theta_{2})}\psi_{\omega}(\cdot+y)\right): \theta_{1}, \theta_{2}, y\in\mathbb{R}\right\}\subset \mathcal{G}_{\gamma, \gamma, \gamma}.
\end{equation*}
Hence the theorem is proved if we can show that any function in $\mathcal{G}_{\gamma, \gamma, \gamma}$ must be of the form given here.
Let $(f_{1}, f_{2}, f_{3})\in \mathcal{G}_{\gamma, \gamma, \gamma}$.  From Lemma \ref{Lolor}(ii),  there exist numbers $\theta_{j}\in \mathbb{R}$ and  a non-negative real function $\rho$ such that $f_{j}(x)=e^{i\theta_{j}}\rho(x)$. Moreover, by the Lagrange multiplier principle there exist numbers $\omega_{1}$, $\omega_{2}$ and $\omega_{3}$ such that 
\begin{equation*}
\begin{cases} 
-\partial^{2}_{x}\rho+2\omega_{1}\rho-(\beta+\alpha)\rho^{2},  \\
-\partial^{2}_{x}\rho+2\omega_{2}\rho-(\beta+\alpha)\rho^{2},  \\
-\partial^{2}_{x}\rho+2\omega_{3}\rho-(\beta+\alpha)\rho^{2},  \\
\end{cases} 
\end{equation*}
It is not difficult to show that $\omega_{1}=\omega_{2}=\omega_{3}>0$. But then, since $\rho\in C^{2}(\mathbb{R})$, an elementary calculation shows that the only real $L^{2}$-solution of system above is given by 
\begin{equation*}
\rho(x)=\frac{3\omega}{(\alpha+\beta)}\mathrm{sech}^{2}\left(\frac{1}{2}\sqrt{2\omega}\,x\right),
\end{equation*}
where $\omega=\omega_{i}$. Therefore, $f_{1}(x)=e^{i\theta_{1}}\psi_{\omega}(x)$, $f_{2}(x)=e^{i\theta_{2}}\psi_{\omega}(x)$ and $f_{3}(x)=e^{i(\theta_{1}+\theta_{2})}\psi_{\omega}(x)$; and hence the result follows.
\end{proof}

\section{Stability of the standing waves}
\label{S:2}

For simplicity, throughout this section we assume that $\beta=1$. Before giving the proof of Theorem \ref{NTE} and Corollary \ref{CLE}, we need to establish some preliminary results.
\begin{lemma}\label{LSdr}
Suppose $\gamma>0$ and  $\mu>0$. The infimum $J(\gamma, \mu)$ defined in \eqref{Nvp} is finite. Furthermore, any minimizing sequence for $J(\gamma, \mu)$ is bounded in $H^{1}(\mathbb{R}; \mathbb{C}^{3})$.
\end{lemma}
\begin{proof}
Suppose that $\left\{\vv{u}_{n}\right\}$ is a minimizing sequence of problem $J(\gamma, \mu)$. Since $\|u_{1, n}\|^{2}_{L^{2}}+\|u_{3, n}\|^{2}_{L^{2}}\rightarrow\gamma$ and  $\|u_{2, n}\|^{2}_{L^{2}}+\|u_{3, n}\|^{2}_{L^{2}}\rightarrow\mu$ as $n$ goes to $+\infty$, it follows that the sequence  $\|u_{i, n}\|^{2}_{L^{2}}$ is bounded for $i=1$, $2$, $3$. Now, from the Gagliardo-Nirenberg inequality and H\"older inequality, it is clear that (see \eqref{E1})
\begin{equation*}
\|u_{i, n}\|^{p+1}_{L^{p+1}}\leq C \|\partial_{x}u_{i, n}\|^{(p-1)/2}_{L^{2}}, \quad \left|\Re\int_{\mathbb{R}}u_{1, n}u_{2, n}\overline{u_{3, n}}\, dx\right|\leq C\sum^{3}_{i=1} \|\partial_{x}u_{i,n}\|^{\frac{1}{2}}_{L^{2}},
\end{equation*}
where $C$ is independent of $u_{i, n}$. Since $p-1<4$, we obtain
\begin{align*}
E(\vv{u}_{n})&=\sum^{3}_{i=1}\left\{\frac{1}{2}\int_{\mathbb{R}}|\partial_{x}u_{i, n}|^{2}\,dx-\frac{1}{p+1}\int_{\mathbb{R}}|u_{i, n}|^{p+1}\,dx\right\}-\alpha\,\Re\int_{\mathbb{R}}u_{1, n}u_{2, n}\overline{u_{3, n}}\, dx \\
&\geq \frac{1}{2}\sum^{3}_{i=1}\|\partial_{x}u_{i, n}\|^{2}_{L^{2}}-C\sum^{3}_{i=1}\|\partial_{x}u_{i, n}\|^{(p-1)/2}_{L^{2}}-C\sum^{3}_{i=1} \|\partial_{x}u_{i,n}\|^{\frac{1}{2}}_{L^{2}}\\
&>-\infty.
\end{align*}
Therefore, $J(\gamma, \mu)>-\infty$. 

The remainder of the proof follows exactly as in Lemma \ref{DL1} i).
\end{proof}

\begin{lemma}\label{LSwer}
Suppose  $\gamma>0$ and  $\mu>0$. If $\left\{(f_{n}, g_{n}, h_{n})\right\}\subseteq H^{1}(\mathbb{R}; \mathbb{C}^{3})$ is a minimizing sequence for $J(\gamma, \mu)$, then there exist a subsequence, which is still denoted by $\left\{(f_{n}, g_{n}, h_{n})\right\}$,  and a number $0< a \leq \min\left\{\gamma, \mu\right\}$ such that 
\begin{equation*}
\lim_{n\rightarrow\infty}\|h_{n}\|_{L^{2}}^{2}=a, \quad \lim_{n\rightarrow\infty}E(f_{n}, g_{n}, h_{n})=I(\gamma-a, \mu-a, a).
\end{equation*}
In particular, $J(\gamma, \mu)=I(\gamma-a, \mu-a, a)$.
\end{lemma}
\begin{proof} To prove this, we use some ideas of \cite{AlbertAngulo, Albert}. Since $\|f_{n}\|^{2}_{L^{2}}+\|h_{n}\|^{2}_{L^{2}}\rightarrow\gamma$ and $\|g_{n}\|^{2}_{L^{2}}+\|h_{n}\|^{2}_{L^{2}}\rightarrow\mu$, it follows that the sequence defined by ${a}_{n}=\|h_{n}\|^{2}_{L^{2}}$ is bounded. Thus,  by passing to a subsequence if necessary,  we can assume that $a_{n}\rightarrow a$. Notice that $0\leq a \leq \min\left\{\gamma, \mu\right\}$. Moreover, $\|f_{n}\|^{2}_{L^{2}}\rightarrow \gamma-a$ and $\|g_{n}\|^{2}_{L^{2}}\rightarrow \mu-a$.

We claim that $a>0$.  Suppose by contradiction that $a=0$.  Then we have that $\|h_{n}\|^{}_{L^{2}}\rightarrow 0$,
$\|f_{n}\|_{L^{2}}\rightarrow \gamma$  and $\|g_{n}\|_{L^{2}}\rightarrow\mu$. In particular, notice that $\|h_{n}\|^{p+1}_{L^{p}}\rightarrow 0$ and $\Re\int_{\mathbb{R}}f_{n}g_{n}\overline{h_{n}}\, dx\rightarrow 0$ as $n$ goes to $+\infty$. Set  $l_{n}={\sqrt{\gamma}}/{\|f_{n}\|_{L^{2}}}$ and   $k_{n}=\sqrt{{\mu}}/{\|g_{n}\|^{}_{L^{2}}}$.
Since $\|l_{n}f_{n}\|^{2}_{L^{2}}=\gamma$ and $\| k_{n}g_{n}\|^{2}_{L^{2}}=\mu$, it follows from Remark \ref{AsDf},
\begin{align*}
J(\gamma, \mu)&=\lim_{n\rightarrow\infty}E(l_{n}f_{n}, k_{n}g_{n}, h_{n})\geq \lim_{n\rightarrow\infty}E_{1}(l_{n}f_{n})+ 
\lim_{n\rightarrow\infty}E_{1}(k_{n}g_{n})\\
&\geq S(\gamma)+S(\mu).
\end{align*}
Now let us consider $\omega_{1}$, $\omega_{2}>0$ such that $\|\phi_{\omega_{1}}\|^{2}_{L^{2}}=\gamma$ and   $\|\phi_{\omega_{2}}\|^{2}_{L^{2}}=\mu$. It is clear that $J(\gamma, \mu)\leq E(\phi_{\omega_{1}}, \phi_{\omega_{2}}, 0)=S(\gamma)+S(\mu)$; therefore, $J(\gamma, \mu)=S(\gamma)+S(\mu)$ and $(\phi_{\omega_{1}}, \phi_{\omega_{2}}, 0)\in \mathcal{M}_{\gamma, \mu}$. Since $(\phi_{\omega_{1}}, \phi_{\omega_{2}}, 0)$ is a minimizer for $J(\gamma, \mu)$, using the Lagrange Multiplier principle, we see that $(\phi_{\omega_{1}}, \phi_{\omega_{2}}, 0)$ has to satisfy \eqref{SP}. In particular, from the last equation of the system  \eqref{SP} we see that $\phi_{\omega_{1}}(x)\phi_{\omega_{2}}(x)=0$ for all $x\in \mathbb{R}$, which is a contradiction. This shows our claim.

Notice that if $\gamma\neq\mu$, then $a<\min\left\{\gamma, \mu\right\}$; the proof is similar to the proof developed above in the case $a=0$. Therefore, if $\gamma\neq\mu$, then $0<a<\min\left\{\gamma, \mu\right\}$.

Next, it is clear that $J(\gamma, \mu)\leq I(\gamma-a, \mu-a, a)$. We claim that 
\begin{equation}\label{mnbv}
J(\gamma, \mu)\geq I(\gamma-a, \mu-a, a).
\end{equation}
To prove \eqref{mnbv}, it suffices to consider the following three cases: (i) $0<a<\min\left\{\gamma, \mu\right\}$; (ii) $\gamma=\mu$ and $0<a<\gamma$; and  (iii) $\gamma=\mu$ and $a=\gamma$.

\textbf{Case (i)} Notice that since $0<a<\min\left\{\gamma, \mu\right\}$, it follows that  $a$, $\gamma-a$, $\mu-a$ are all positive.
Then, for $n$ sufficiently large we may define
\begin{align*}
l_{n}&=\frac{\sqrt{\gamma-a}}{\|f_{n}\|^{}_{L^{2}}},\quad\quad  k_{n}=\frac{\sqrt{\mu-a}}{\|g_{n}\|^{}_{L^{2}}}, \quad \quad b_{n}=\frac{\sqrt{a}}{\|h_{n}\|^{}_{L^{2}}},
\end{align*}
then we see  that $l_{n}$, $k_{n}$, $b_{n}\rightarrow 1$ as $n$ goes to $+\infty$ and 
\begin{align*}
\|l_{n}f_{n}\|^{2}_{L^{2}}=\gamma_{}-a,\quad\quad  \| k_{n}g_{n}\|^{2}_{L^{2}}=\mu_{}-a, \quad\quad\|b_{n}h_{n}\|^{2}_{L^{2}}=a.
\end{align*}
Therefore, 
\begin{align*}
J(\gamma, \mu)=\lim_{n\rightarrow\infty}E(l_{n}f_{n}, k_{n}g_{n}, b_{n}h_{n})\geq I(\gamma-a, \mu-a, a),
\end{align*}
and hence \eqref{mnbv} follows. 

\textbf{Case (ii)} The proof is the same as in the Case (i). 

\textbf{Case (iii)} In this case we have that $\|f_{n}\|^{}_{L^{2}}\rightarrow 0$, $\|g_{n}\|_{L^{2}}\rightarrow 0$  and $\|h_{n}\|_{L^{2}}\rightarrow \gamma$. This implies that $\|f_{n}\|^{p+1}_{L^{p}}\rightarrow 0$, $\|g_{n}\|^{p+1}_{L^{p}}\rightarrow 0$ and $\Re\int_{\mathbb{R}}f_{n}g_{n}\overline{h_{n}}\, dx\rightarrow 0$. Therefore,
\begin{align*}
J(\gamma, \gamma)&=\lim_{n\rightarrow\infty}E(f_{n}, g_{n}, c_{n}h_{n})\geq \lim_{n\rightarrow\infty}E_{1}(c_{n}h_{n})\\
&\geq S(\gamma)=I(0, 0, \gamma),
\end{align*}
where $c_{n}=\sqrt{{\gamma}}/{\|h_{n}\|^{}_{L^{2}}}$. This shows our claim in all cases and the proof of the lemma is completed.
\end{proof}

Now we give the proof of Theorem \ref{NTE}.
\begin{proof}[\bf {Proof of Theorem \ref{NTE}}] 
Let $\left\{(f_{n}, g_{n}, h_{n})\right\}$ be a minimizing sequence for $J(\gamma, \mu)$. 
We claim that the minimizing sequence $\left\{(f_{n}, g_{n}, h_{n})\right\}$  is relatively compact in $H^{1}(\mathbb{R}; \mathbb{C}^{3})$ up to translations. Indeed, from Lemma \ref{LSwer} and  passing to a subsequence,  we may assume that $\left\{(f_{n}, g_{n}, h_{n})\right\}$ is a minimizing sequence for $I(\gamma-a, \mu-a, a)$, with $0<a \leq\min\left\{\gamma, \mu\right\}$. Now, if $0 < a < \min\left\{\gamma, \mu\right\}$, we have that $a$, $\gamma-a$, $\mu-a$ are all positive, then  Theorem \ref{SPI} (i) allows us to conclude that there exist $\left\{y_{n}\right\}\subset \mathbb{R}$ and an element $\vv{\varphi}$ such that  $\left\{(f_{n}(\cdot+y_{n}), g_{n}(\cdot+y_{n}), h_{n}(\cdot+y_{n}))\right\}$ has a subsequence converging  in $H^{1}(\mathbb{R}; \mathbb{C}^{3})$ to $\vv{\varphi}$.

If, on the other hand, $\gamma=\mu$  and $a=\min\left\{\gamma, \mu\right\}=\gamma$, then as in the proof of Case (ii) in  Lemma \ref{LSwer} we see that
\begin{align*}
I(0, 0, \gamma)&=J(\gamma, \mu)=\lim_{n\rightarrow\infty}E(f_{n}, g_{n}, h_{n})\\
&\geq\lim_{n\rightarrow\infty}E_{1}(h_{n})\geq I(0, 0, \gamma).
\end{align*}
Therefore,  $\|h_{n}\|^{2}_{L^{2}}\rightarrow \gamma$ and $E_{1}(h_{n})\rightarrow S(\gamma)$. Then the claim is easily deduced from Remark \ref{AsDf}. Thus we obtain the proof of (i) of the theorem. 

Statement (ii) of the theorem is obvious from the definition of the sets $\mathcal{M}_{\gamma, \mu}$. Next, we prove (iii). Suppose $(f, g, h)\in \mathcal{M}_{\gamma, \mu}$. From Lemma \ref{LSwer} we have that  
$(f, g, h)\in \mathcal{G}_{\gamma-a, \mu-a,a}$ for some $0< a \leq \min\left\{\gamma, \mu\right\}$. Again, 
if $0< a <\min\left\{\gamma, \mu\right\}$ then the statement  (iii) follows immediately from part (iv) of the Theorem \ref{SPI}.
On the other hand, if $\gamma=\mu$ and $a=\gamma$,  the statement  (iii) is an  immediate consequence of the Remark \ref{AsDf}. This completes the proof of the theorem.
\end{proof}

Now we give the proof of Corollary \ref{CLE}.
\begin{proof}[\bf {Proof of Corollary \ref{CLE}}]
Suppose that $\mathcal{M}_{\gamma, \mu}$ is $H^{1}(\mathbb{R}; \mathbb{C}^{3})$-unstable. Then there are some $\epsilon>0$, initial data $\vv{u}_{0, n}$ and points $t_{n}>0$ such that 
\begin{equation*}
\inf\left\{\|\vv{u}_{0 ,n}-\vv{\varphi}\|_{H^{1}(\mathbb{R}; \mathbb{C}^{3})}: \vv{\varphi}\in \mathcal{M}_{\gamma, \mu}\right\}<\frac{1}{n},
\end{equation*}
but
\begin{equation}\label{bnj}
\inf\left\{ \|\vv{u}_{n}(t_{n})-\vv{\varphi}\|_{H^{1}(\mathbb{R}; \mathbb{C}^{3})}:  \vv{\varphi}\in \mathcal{M}_{\gamma, \mu}  \right\}\geq {\epsilon},
\end{equation}
for all $n$, where $\vv{u}_{n}(x,t)$ denotes the solution of the Cauchy problem \eqref{NLS} with initial data $\vv{u}_{0, n}$. Since $\vv{u}_{0, n}\rightarrow \vv{\varphi}$ in $H^{1}(\mathbb{R}; \mathbb{C}^{3})$ as $n$ goes to $+\infty$, and since $Q_{1}(\vv{\varphi})=\gamma$, $Q_{2}(\vv{\varphi})=\mu$, we see that
\begin{equation}\label{Sdg}
\lim_{n\rightarrow\infty} Q_{1}(\vv{u}_{0, n})=\gamma, \quad \lim_{n\rightarrow\infty} Q_{2}(\vv{u}_{0, n})=\mu \quad, \lim_{n\rightarrow\infty} E(\vv{u}_{0, n})=J(\gamma, \mu).
\end{equation}
Since $E$, $Q_{1}$ and $Q_{2}$ are independent of $t$, from \eqref{Sdg} we obtain
\begin{equation*}
\lim_{n\rightarrow\infty} Q_{1}(\vv{u}_{n}(t_{n}))=\gamma, \quad \lim_{n\rightarrow\infty} Q_{2}(\vv{u}_{n}(t_{n}))=\mu, \quad \lim_{n\rightarrow\infty} E(\vv{u}_{n}(t_{n}))=J(\gamma, \mu).
\end{equation*}
Thus,  we have that $\left\{\vv{u}_{n}(t_{n})\right\}$ is a minimizing sequence for $J(\gamma, \mu)$.  By part (i) of the Theorem \ref{NTE}, up to a subsequence, there exist a sequence $\left\{y_{n}\right\}\subset\mathbb{R}$ and a function  $\vv{\psi}\in \mathcal{M}_{\gamma,\mu}$ such that 
\begin{gather*}
\lim_{n\rightarrow\infty}\|\vv{u}_{n}(\cdot+y_{n}, t_{n})-\vv{\psi}\|_{H^{1}(\mathbb{R}; \mathbb{C}^{3})}=0.
\end{gather*}
Since $\vv{\psi}(\cdot-y_{n})\in \mathcal{M}_{\gamma,\mu}$, we see that for all sufficiently large $n$,
\begin{gather*}
\inf\left\{ \|\vv{u}_{n}(t_{n})-\vv{\varphi}\|_{H^{1}(\mathbb{R}; \mathbb{C}^{3})}:  \vv{\varphi}\in \mathcal{M}_{\gamma,\mu}      \right\}<{\epsilon}{},
\end{gather*}
which is a contradiction with \eqref{bnj}. This finishes the proof.
\end{proof}

\section*{Acknowledgements}
The author wishes to express his sincere thanks to the referees for their valuable
comments. He also gratefully acknowledges financial support from CNPq,
through grant No. 152672/2016-8.

\bibliographystyle{plain}
\bibliography{bibliografia}


\end{document}